\documentclass[final]{siamart0516}

\usepackage{todonotes}
\newcommand{\Mr}{\mathcal{M}_{\le r}}
\newcommand{\wMr}{\mathcal{\widehat{M}}_r}
\usepackage{lipsum}
\usepackage{amsfonts}
\usepackage{graphicx}
\usepackage{epstopdf}
\usepackage{algorithmic}
\usepackage{amsopn}
\usepackage{mathtools}
\usepackage{breqn}
\usepackage{graphicx}
\usepackage{subcaption}
\usepackage[multiple]{footmisc}
\usepackage{float}
\newcommand{\specificthanks}[1]{\@fnsymbol{#1}}
\newtheorem*{remark}{Remark}

\author{
	Valentin Khrulkov\thanks{Skolkovo Institute of Science and Technology 143025, Nobel St. 3, Skolkovo Innovation Center Moscow, Russia.}
	\and
	Ivan Oseledets\footnotemark[1] \thanks{Institute of Numerical Mathematics of the Russian Academy of Sciences 119333, Gubkina St. 8 
			Moscow, Russia .}
}

\title{Desingularization of bounded-rank matrix sets}
\begin{document}
\maketitle
\begin{abstract}
The conventional ways to solve optimization problems on low-rank matrix sets which appear in a great number of applications tend to ignore its underlying structure of an algebraic variety and existence of singular points. This leads to the appearance of inverses of singular values in algorithms and since they could be close to $0$ it causes certain problems. We tackle this problem by utilizing ideas from algebraic geometry and show how to desingularize these sets. Our main result is an algorithm which uses only bounded functions of singular values and hence does not suffer from the issue described above. 
\end{abstract}
\begin{keywords}
	low-rank matrices, algebraic geometry, Riemannian optimization, matrix completion
\end{keywords}

\begin{AMS}
	65F30
\end{AMS}
\section{Introduction}
Although low-rank matrices appear in many applications, the structure of the corresponding matrix variety (real algebraic) is not fully utilized in the computations, and the theoretical investigation is complicated because of the existence of singular points \cite{lakshmibai2015grassmannian} on such a variety, which correspond to matrices of smaller rank. We tackle this problem by utilizing the modified Room-Kempf desingularization \cite{naldi2015exact} of determinantal varieties that
is classical in algebraic geometry, but has never been applied in the context of optimization over matrix varieties.
 Briefly, it can be summarized as follows. Idea of the the Room-Kempf procedure is to consider a set of tuples of matrices $(A,Y)$ satisfying equations $AY=0$ and $BY=0$ for some fixed matrix $B$. These equations imply that the rank of $A$ is bounded and moreover a set of such tuples is a smooth manifold (for reasonable matrices $B$). However, conditions of the form $BY=0$ can be numerically unstable, so we modify it by imposing the condition $Y^T Y = I$ instead. The precise definition of the manifold we work with is given in terms of Grassmannians and then we transition to the formulas given above. We also show that the dimension of this manifold is the same as of the original matrix variety.
Our main contributions are:

\begin{itemize}
    \item We propose and analyze a modified Room-Kempf desingularization technique for the variety of matrices of shape $n \times m$ with rank bounded by $r$ (\cref{notsing:secdes}). 
    \item We prove smoothness and obtain bounds on the curvature of the desingularized variety in \cref{notsing:secdes} and \cref{sec:ts-desing}. The latter is performed by estimating singular values of the operator of the orthogonal projection onto the tangent space of the desingularized variety. 
    \item We find an effective low-dimensional parametrization of the tangent space (\cref{notsing:tspar}). Even though the desingularized variety is a subset of a space of much bigger dimension, this allows us to construct robust second order method with $O((n+m)r)$ complexity.  
    \item We implement an effective realization of a reduced Hessian method for the optimization over the desingularized variety (\cref{notsing:secnewton}). We start with the Lagrange multipliers method for which we derive a formula for the Newton method for the corresponding optimization problem. The latter takes the saddle point form which we solve using the null space method found in \cite{benzi2005numerical}. In \cref{sec:trick} we show how to reduce the total complexity of the algorithm to $O((n+m)r)$ per iteration.
    \item We also briefly discuss a few technical details in the implementation of the algorithm (\cref{notsing:techstuff})
    \item We present results of numerical experiments and compare them with some other methods found in \cref{notsing:numerical}. 
\end{itemize}

The manifolds that we work with in this paper will always be $C^{\infty}$ and in fact smooth algebraic varieties.
\subsection{Idea of desingularization} \label{notsing:intro}

Before we define desingularization of bounded rank matrix sets, we will introduce its basic idea. The low-rank matrix case will be described in next section.
Let $V$ be a variety (not necessarily smooth) and $f$ be a function
$$f : V \to \mathbb{R}, $$

which is smooth in an open neighborhood of $V$ (which is assumed to be embedded in $\mathbb{R}^{k}$).
To solve 
$$ f(x) \to \min,  x \in V,$$
we often use methods involving the tangent bundle of $V$. However, due to the existence of the singular points where the tangent space is not well-defined, it is hard to prove correctness and convergence using those methods. To avoid this problem we construct a smooth variety $\widehat{V}$ and a surjective smooth map $\pi$
$$\pi: \widehat{V} \to V.$$
Let $\widehat{f}$ be a pullback of $f$ via map $\pi$ i.e.
$$ \widehat{f} : \widehat{V} \to \mathbb{R},$$
$$ \widehat{f} = f \circ \pi.$$
It is obvious that 
$$ \min_{x \in V} f(x) = \min_{y \in \widehat{V}} \widehat{f}(y),$$
so we reduced our non-smooth minimization problem to a smooth one. Typically $\widehat{V}$ is a variety in a space of bigger dimension and is constructed to be of the same dimension as the smooth part of $V$.
To have some geometrical idea one can think about the following example (see \cref{fig:bu}).
Let $V$ be a cubic curve given by the following equation
$$y^2 = x^2 (x+1),$$
and parametrized as
$$(x(t),y(t)) = (t^2-1,t(t^2-1)).$$
It is easy to see that $(0,0)$ is a singular point of $V$.
Then its desingularization is given by 
$$\widehat{V} = (x(t),y(t),z(t)) = (t^2-1,t(t^2-1),t) \subset \mathbb{R}^3,$$ which is clearly smooth. Projection is then just 
$$ \pi: (x(t),y(t),z(t)) = (x(t),y(t)).$$ 

\begin{figure}[h] 
	\begin{subfigure}{.4\textwidth}
		\centering
		\includegraphics[width=.5\linewidth]{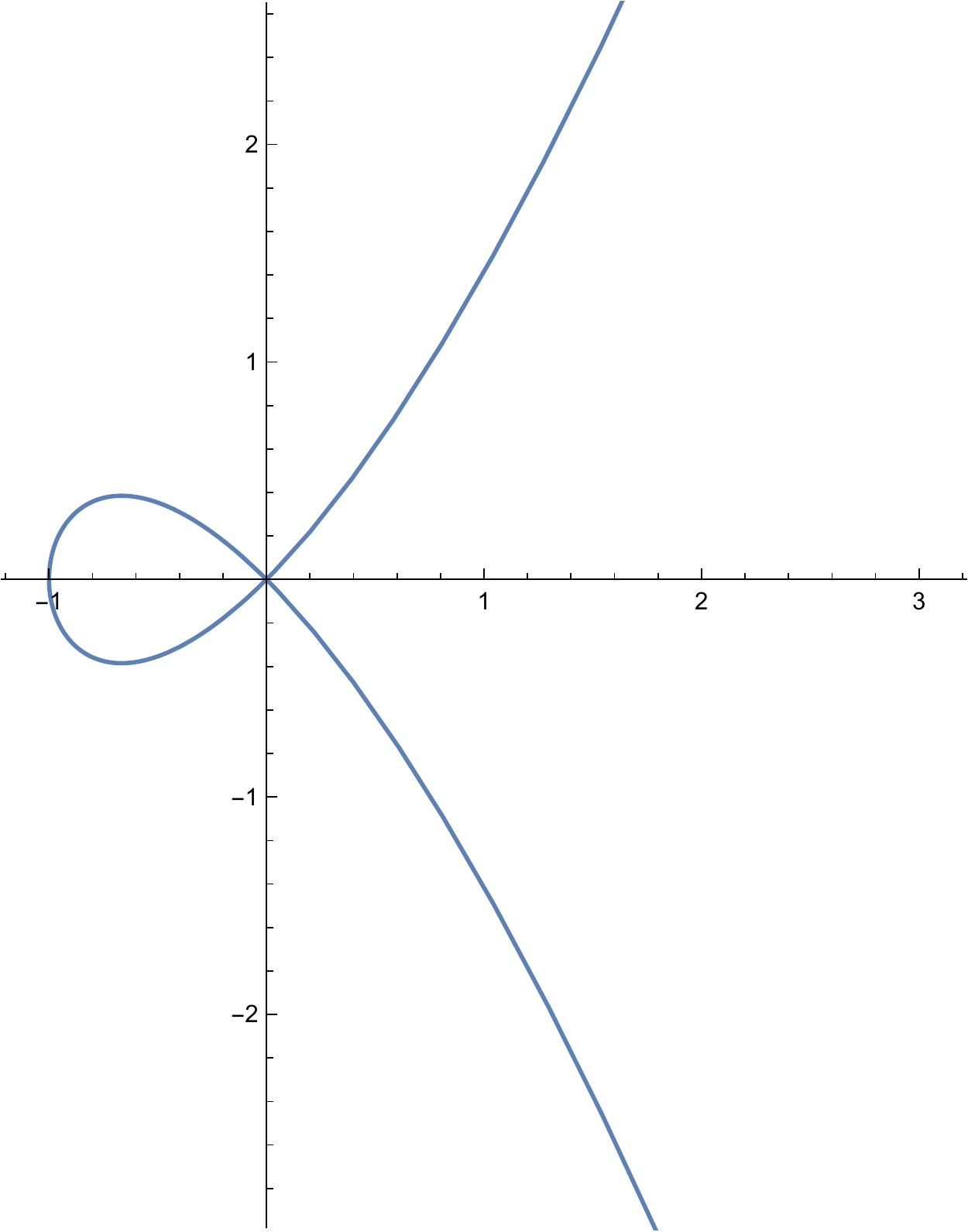}
        \subcaption{Singular cubic}
	\end{subfigure}
	\begin{subfigure}{.4\textwidth}
		\centering
		\includegraphics[width=.5\linewidth]{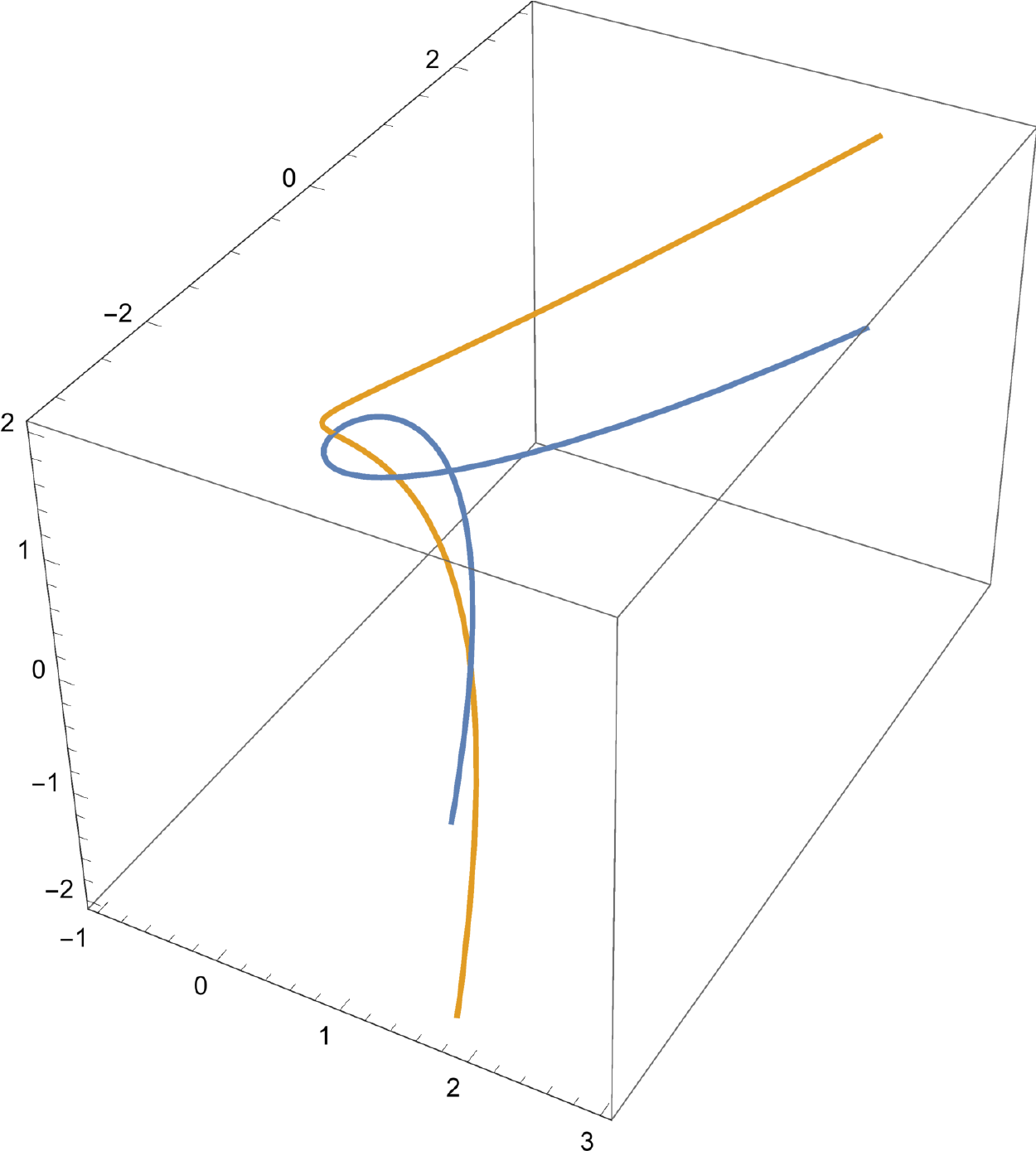}
		\subcaption{Desingularized cubic}
	\end{subfigure}

	\caption{Desingularization of the cubic.}
	\label{fig:bu}
\end{figure}

\section{Desingularization of low-rank matrix varieties via kernel} 
\subsection{$2\times 2$ matrices} \label{notsing:2x2sec}
Let $V$ be a variety of $2 \times 2$ matrices with the rank $\leq 1$. We have 
\begin{equation}\label{notsing:2x2sing}
V = \lbrace (x_{11},x_{21},x_{12},x_{22})  \in \mathbb{R}^4  : x_{11}x_{22}-x_{12}x_{21} = 0 \rbrace ,
\end{equation}
so it is indeed an algebraic variety. In order to analyze its smoothness and compute the tangent space we recall the following result. 

Let $h_i \quad i \in \lbrace 1 \hdots k \rbrace $ be some smooth functions 
$$ h_i : \mathbb{R}^l \to \mathbb{R},$$ with $k \leq l.$
Define the set $M$ as 
$$ M = \lbrace x \in \mathbb{R}^{l} : h_1(x)=0, h_2(x)=0  \ldots h_k(x)=0  \rbrace.$$
Then for a point $p \in M$
we construct the matrix $N(p)$,
$$N(p) = 
\begin{bmatrix}
\nabla h_1(p) \\
\nabla h_2(p) \\
\vdots \\
\nabla h_k(p)
\end{bmatrix},
$$
where $\nabla h_i(p)$ is understood as the row vector
$$\nabla h_i(p) = \left( \frac{\partial h_i}{\partial x_1}, \hdots, \frac{\partial h_i}{\partial x_l} \right).$$
A point $p$ is called nonsingular if $N(p)$ has maximal row rank at $p$. In this case, by implicit function theorem, $M$ is locally a manifold (see \cite[Theorem 5.22]{lee-introman-2001}) and tangent space at $p$ is defined as  
$$T_p M = \lbrace v \in \mathbb{R}^l : N(p)v = 0 \rbrace .$$
Applying this to $V$ defined in \cref{notsing:2x2sing} we obtain
$$N(x_{11},  x_{21}, x_{12},  x_{22}) = \begin{bmatrix}
x_{22} & -x_{12} & -x_{21} & x_{11}
\end{bmatrix},
$$
and then $(0,0,0,0)$ is a singular point of $V$. 

We desingularize it by considering $\widehat{V}$ which is defined as the set of pairs $(A,Y) \in \mathbb{R}^{2 \times 2} \times \mathbb{R}^2$ with coordinates
$$A = 
\begin{bmatrix}
x_{11} & x_{12} \\
x_{21} & x_{22}
\end{bmatrix},
$$ and 
$$
Y = \begin{bmatrix}
y_1\\
y_2
\end{bmatrix},
$$ 
satisfying $$AY = 0,$$ and $$Y^{\top} Y = 1.$$ 
Such choice of equations for $Y$ is based on the Room-Kempf procedure described in \cite{naldi2015exact}, which suggests the following equations:
$$AY=0, \quad BY=0,$$ with some fixed matrix $B$. Since the latter equation is numerically unstable, using an orthogonality condition instead allows us to maintain the manifold property while making computations more robust.

More explicitly we have
$$\widehat{V} =  \lbrace p: (x_{11}y_1+x_{12}y_2=0,x_{21}y_1+x_{22}y_2=0,y_1^2+y_2^2=1)\rbrace,
$$
$$
p = (x_{11},  x_{21}, x_{12},  x_{22}, y_1, y_2) \in \mathbb{R}^6.
$$
We find that the normal space at $p$ is spanned by rows of the following matrix $N(p)$:
   \begin{equation}\label{notsing:n2nullspace}
   N(p) = \begin{bmatrix} y_1 & 0 & y_2 & 0 &x_{11} & x_{12}\\
   0 & y_1 & 0 & y_2 & x_{21} & x_{22}\\
   0 & 0 & 0 & 0 & 2y_1&2y_2  \\
   \end{bmatrix}.
   \end{equation}

Since $y_1^2+y_2^2=1$ the matrix \cref{notsing:n2nullspace} clearly has rank $3$ at any point of $\widehat{V}$ which proves that $\widehat{V}$ is smooth. The projection $\pi$ is just $$\pi : (x_{11},  x_{21}, x_{12},  x_{22}, y_1, y_2) \to (x_{11},  x_{21}, x_{12},  x_{22}),$$ whose image is the entire $V$. 
However, we would also like to estimate how close the tangent spaces are at close points. Recall that by definition of the Grassmanian metric the distance between subspaces $C$ and $D$ is given by
$$d_{Gr}(C,D) := \|P_{C} - P_{D}\|_F,$$
where $P_{C}$ and $P_{D}$ are the orthogonal projectors on the corresponding planes.
Since {$P_{C^{\perp}} = I - P_C$ the distance between any two subspaces is equal to the distance between their orthogonal complements. 
\\
It is well known that the projection on the subspace spanned by the rows of a matrix $M$ is given by 
$M^{\dagger}M$, where $M^{\dagger}$ is a pseudoinverse which for matrices of maximal row rank is defined as $$M^{\dagger} =M^{\top} (M M^{\top})^{-1}.$$
Hence, for two different points $p$ and $p'$ on the desingularized manifold we obtain
$$ \|P_{N(p)} - P_{N(p')}\|_F = \| N(p)^{\dagger}N(p) - N(p')^{\dagger} N(p') \|_F  .$$
We will use the following classical result to estimate $\|P_{N(p)} - P_{N(p')}\|_F$ (we use it in the form appearing in \cite[Lemma~3.4]{dutta2017problem} which is based on the \cite[The $\sin \theta$ Theorem]{davis1970rotation}):
\begin{equation} \label{notsing:projectorbound}
\| N(p)^{\dagger}N(p) - N(p')^{\dagger} N(p') \|_F  \leq 2\max \lbrace \| N(p)^{\dagger} \|_2,
\| N(p')^{\dagger} \|_2 \rbrace \| N(p) - N(p') \|_F.
\end{equation}
In order to estimate the smoothness we need to estimate how $P_{N(p)}$ changes under small changes of $p$. It is sufficient to estimate the gradient of $P$.
Thus, we have to uniformly bound $\Vert N^{\dagger}\Vert_2$ from above, which is equivalent to bounding the minimal singular value of $N$ from below. Denote the latter by $\sigma_{\min}(N)$. 
By taking the defining equations of the desingularized manifold into account, we find that
   \begin{equation}\label{notsing:n2sing}
   N(p) N(p)^{\top} = \begin{bmatrix} 1+x_{11}^2+x_{12}^2 & x_{11} x_{21}+x_{12} x_{22}& 0 \\
   x_{11} x_{21}+x_{12} x_{22} & 1+x_{21}^2+x_{22}^2 & 0 \\
   0 & 0 & 4 
   \end{bmatrix}.
   \end{equation}
Hence $\sigma_{\min}^2 (N(a)) \geq 1$ and $\| N(a)^{\dagger} \|_2 \leq 1 $.
From the definition of $N(p)$ it follows that for $p=(A,Y)$ and $ p'=(A', Y')$: $$\| N(p) - N(p') \|_F \leq \sqrt{6}\| Y - Y' \|_F + \| A - A' \|_F.$$ and from \cref{notsing:projectorbound} we obtain
$$d_{Gr}(T_p \widehat{V}, T_{p'} \widehat{V}) \leq 2 \sqrt{6} (\|A-A'\|_F + \|Y-Y'\|_F).$$ 
We will derive and prove similar estimates for the general case in the next section.
\subsection{General construction and estimation of curvature} \label{notsing:secdes}
\begin{remark}\label{remark}
We will often use vectorization of matrices which is a linear operator
$$\mathrm{vec} : \mathbb{R}^{m \times n} \to \mathbb{R}^{mn \times 1},$$
which acts by stacking columns of the matrix into a single column vector. To further simplify notation, variables denoted by uppercase and lowercase variables are understood as a matrix and vectorization of the corresponding matrix, e.g. $p =\mathrm{vec}({P}).$ 
We will also define the transposition operator $T_{m,n}$:
$$T_{m,n} : \mathrm{vec}(X) \to \mathrm{vec}(X^{\top}),$$
for $X \in \mathbb{R}^{m \times n}$.
\end{remark}
Consider a variety $\Mr$ of $n \times m$ of matrices of rank not higher than $r$, 
$$\Mr = \{ A \in \mathbb{R}^{n \times m}: \mathrm{rank}(A) \leq r \}.$$

We recall the following classical result \cite[Theorem 10.3.3]{lakshmibai2015grassmannian}.
\begin{lemma}
$A \in \Mr$ is a singular point if and only if $A$ has rank smaller than $r$.
\end{lemma}
By definition, the dimension of a variety $X$ is equal to the dimension of the manifold $X \setminus X_{sing}$  where $X_{sing}$ is the set of all singular points of $X$ \cite{griffiths2014principles}. In the case of $\Mr$ we find that 
$$\dim \Mr = \dim \mathcal{M}_{=r},$$
where $$\mathcal{M}_{=r} =\{ A \in \mathbb{R}^{n \times m}: \mathrm{rank}(A) = r \},$$ is known to be of dimension $(n+m)r - r^2$ (e.g. \cite[Proposition 2.1]{vandereycken2013low}).

Now we return to the main topic of the paper. 
\\
Let $Gr(m-r,m)$ be the Grassmann manifold:
$$Gr(m-r,m) = \mathbb{R}_{*}^{m,m-r} / GL_{m-r},$$
where $\mathbb{R}_{*}^{m,m-r}$ is the noncompact Stiefel manifold
$$\mathbb{R}_{*}^{m,m-r}  = \lbrace Y \in \mathbb{R}^{m \times (m-r)}: Y \text{ full rank} \rbrace, $$
and $GL_{m-r}$ is the group of invertible $m-r \times m-r$ matrices.
\\
It is known \cite{lee-introman-2001} that $$\dim Gr(m-r,m) = r(m-r).$$

We propose the following desingularization for $\wMr$:
\begin{equation}
\label{notsing:grassm}
\wMr = \lbrace (A,Y) \in \mathbb{R}^{n \times m} \times Gr(m-r,m)  : AY=0 \rbrace ,
\end{equation}
and prove the following theorem.
\begin{theorem}
$\wMr$ as defined by \cref{notsing:grassm} is a smooth manifold of dimension $(n+m)r - r^2$.
\end{theorem}
\begin{proof}
Let $U_{\alpha}$ be a local chart of $Gr(m-r,m)$. 
To prove the theorem it suffices to show that $\wMr \cap( \mathbb{R}^{n \times m} \times U_{\alpha})$ is a smooth manifold for all $\alpha$. Without loss of generality let us assume that the coordinates of $Y \in Gr(m-r,m) \cap U_\alpha$ are given by 
$$
Y = \begin{bmatrix}
I_{m-r} \\
Y_{\alpha} \\
\end{bmatrix},
$$
where
$$
Y_{\alpha} = 
\begin{bmatrix}
\alpha_{1,1} & \alpha_{1,2} & \hdots &\alpha_{1,m-r} \\
\alpha_{2,1} & \alpha_{2,2} & \hdots &\alpha_{2,m-r} \\
\hdots & \hdots & \hdots & \hdots \\
\alpha_{r,1} & \alpha_{2,2} & \hdots &\alpha_{r,m-r} \\
\end{bmatrix}.
$$
In this chart equation \cref{notsing:grassm} reads 
\begin{equation} \label{eq:local-grassm}
A \begin{bmatrix}
I_{m-r} \\
Y_{\alpha} \\
\end{bmatrix} = 0.
\end{equation}
Splitting $A$ as 
$$A = 
\begin{bmatrix}
A_1 & A_2
\end{bmatrix},
$$
where $$A_1 \in \mathbb{R}^{n \times (m-r)}, \quad A_2 \in \mathbb{R}^{n \times r},$$ and by using properties of the Kronecker product $\otimes$ we obtain that the Jacobian matrix of \cref{eq:local-grassm} is equal to
$$ 
\begin{bmatrix}
I_{n(m-r)} & Y_{\alpha}^{\top} \otimes I_n & I_{m-r} \otimes A_2 \\
\end{bmatrix},
$$
which is clearly of full rank, since it contains identity matrix.
To conclude the proof we note that 
$$\dim \wMr = \underbrace{nm + (m-r)r}_{\text{number of variables}} -\underbrace{n(m-r)}_{\text{number of equations}} = (n+m)r - r^2,$$
as desired.
\end{proof}
The use of $\wMr$ is justified by the simple lemma
\begin{lemma}
The following statements hold:
\begin{itemize}
\item If $(A, Y) \in \wMr$ then $A \in \Mr$,
\item If $A \in \Mr$ then there exists $Y$ such that $(A,Y) \in \wMr$.
\end{itemize}
\end{lemma}
\begin{proof}
These statements obviously follow from the equation
$$AY = 0,$$
which implies that the dimension of the nullspace of $A$ is at least $m-r$.
\end{proof}
 We would like to construct Newton method on the manifold $\wMr$. In order to work with quotient manifolds such as $Gr(m-r,m)$ the conventional approach is to use the total space of the quotient. The tangent space is then dealt with using the concept of \emph{horizontal space} (sometimes this is referred to as \emph{gauge condition}) which is isomorphic to the tangent space of the quotient manifold. This approach is explained in great detail in \cite{absil2009optimization}. Although we will not go into the details of these concepts, we will apply them to $\wMr$ in the next section.
\subsection{Tangent space of $\wMr$}
\label{sec:ts-desing}
For our analysis, it is more convenient to employ the following representation of the Grassmanian:
\begin{equation}
\label{eq:grassm-stief}
Gr(m-r,m) = St(m-r,m)/O_{m-r},
\end{equation}
where $St(m-r,m)$ is the orthogonal Stiefel manifold
$$St(m-r,m) = \lbrace Y \in \mathbb{R}^{m,m-r}_{*}: Y^{\top} Y = I_{m-r} \rbrace,$$
and $O_{m-r}$ is the orthogonal group. 
}

Let $\pi$ be the quotient map \cref{eq:grassm-stief} and $\text{id} \times \pi$ is the obvious map
$$\mathbb{R}^{n \times m} \times St(m-r,m) \to \mathbb{R}^{n \times m} \times Gr(m-r,m).$$ It is easy to see that $\wMr^{\text{tot}} \coloneqq (\text{id} \times \pi)^{-1} (\wMr)$ is the following manifold:
\begin{equation}
\label{eq:defining}
\wMr^{\text{tot}} = \lbrace (A,Y) \in \mathbb{R}^{n \times m} \times \mathbb{R}^{m, m-r}_{*} : A Y = 0,\quad  Y^{\top} Y = I_{m-r} \rbrace. 
\end{equation}

Let us now compute the horizontal distribution on $\wMr^{\text{tot}}$. As described in \cite[Example 3.6.4]{absil2009optimization} in the case of the projection $$\pi: \mathbb{R}_{*}^{m, m-r} \to Gr(m-r,m),$$
$$ Y \to \text{span} (Y),$$
the horizontal space at $Y$ is defined as the following subspace of $T_Y \mathbb{R}_{*}^{m, m-r}$:
\begin{equation}\label{eq:gauge}
\lbrace \delta Y \in T_Y \mathbb{R}_{*}^{m,m-r} : (\delta Y)^{\top} Y = 0 \rbrace .
\end{equation}
It immediately follows that in the case \cref{eq:defining} the horizontal space at $(A,Y)$ is equal to
$$\mathcal{H}(A,Y) = T_{(A,Y)}(\wMr^{\text{tot}}) \cap \mathcal{H}_{\text{Gr}}(A,Y),$$
where $\mathcal{H}_{\text{Gr}}(A,Y)$ is similarly to \cref{eq:gauge} defined as:
\begin{equation}\label{eq:true-gauge}
\mathcal{H}_{\text{Gr}}(A,Y) = \lbrace (\delta A, \delta Y) \in T_{(A,Y)}(\mathbb{R}^{n \times m}\times \mathbb{R}_{*}^{m, m-r}) :  (\delta Y)^{\top} Y = 0 \rbrace.
\end{equation}
Note that the dimension of $\mathcal{H}$ is equal to the dimension of $\wMr$ since it is, by construction, isomorphic to the $T\wMr$.
We now proceed to one of the main results of the paper
\begin{theorem}
The orthogonal projection on $\mathcal{H}(A,Y)$ is Lipschitz continuous with respect to $(A,Y)$ and its Lipschitz constant is no greater than $2 (\sqrt{n} + \sqrt{m-r})$ in the Frobenius norm.
\end{theorem}
\begin{proof}
In order to prove the theorem, first we need to find the equations of $\mathcal{H}(A,Y)$.  Recall the defining equations of $\wMr^{\text{tot}}$ \cref{eq:defining} and that for a given $p=(A, Y)$ the tangent space is the nullspace of the gradient of the constraints. By taking into account the gauge condition \cref{eq:true-gauge}
we find that 
    $$\mathcal{H}(A,Y) = \{v: N(p) v = 0 \},$$
    where the matrix $N(p)$ 
   has the following block structure:
   \begin{equation}\label{notsing:nullspace2}
      N(p) = \begin{bmatrix} Y^{\top} \otimes I_n & I_{m-r} \otimes A \\
       0 & I_{m-r} \otimes Y^{\top}
   \end{bmatrix}.
   \end{equation}
   For simplicity of notation we will omit $p$ in $N(p)$.
   The projection onto the horizontal space of a given vector $z$ is given by the following formula 
   \begin{equation}\label{notsing:tangent2}
   v = (I - N^{\top} (N N^{\top})^{-1} N) z = P_{N} z,
   \end{equation}
   where 
   $$P_{N} = (I - N^{\dagger} N ),$$
   is the orthogonal projector onto the row range of $N$.
Using exactly the same idea as in previous section we estimate
   $\sigma_{\min}(N)$ from below.
  Consider the Gram matrix
  \begin{equation*}
  \begin{split}
  Z = N N^{\top} = \begin{bmatrix} Y^{\top} \otimes I_{n} & I_{m-r} \otimes A \\ 
  0   & I_{m-r} \otimes Y^{\top}
                  \end{bmatrix}
                  \begin{bmatrix}
                      Y  \otimes I_n & 0 \\
                      I_{m-r} \otimes A^{\top}& I_{m-r} \otimes Y
                  \end{bmatrix} = \\
                  =\begin{bmatrix} Y^{\top} Y \otimes I_n +   I_{m-r} \otimes A A^{\top} &  I_{m-r} \otimes AY \\
                  I_{m-r} \otimes Y^{\top} A^{\top} &  I_{m-r} \otimes Y^{\top} Y
                  \end{bmatrix}.
  \end{split}
  \end{equation*}
  Now we recall that for each point at the manifold $\wMr^{\text{tot}}$ \cref{eq:defining} holds, therefore 
  \begin{equation}\label{notsing:zmatrix}
      Z = \begin{bmatrix} I +  I_{m-r} \otimes AA^{\top} & 0 \\
  0 &  I \end{bmatrix}.
  \end{equation}
  It is obvious that $\sigma_{\min}(Z) \geq 1$ since it has the form $I + DD^{\top}$.
  Finally, $\sigma^2_{\min}(N) = \sigma_{\min}(Z) \geq 1$, therefore
  \begin{equation}\label{notsing:sigmabound}
      \sigma_{\min}(N) = \sigma_{\min}(N^{\top}) \geq 1, \quad \Vert (N^{\top})^{\dagger} \Vert_2 \leq 1.
  \end{equation}
  Putting \cref{notsing:sigmabound} into \cref{notsing:projectorbound} we get
  $$
    \Vert P_N - P_{N'} \Vert_F \leq  2 \Vert N - N' \Vert_F,
  $$
   with $N = N(A,Y), N' = N(A', Y')$.
  Finally, we need to estimate how $N$ changes under the change of $A$ and $Y$.
 We have
  $$
     N - N' = \begin{bmatrix} (Y - Y') \otimes I_n &  I_{m-r} \otimes (A - (A')) \\
0 & I_{m-r} \otimes (Y^{\top} - (Y')^{\top}) 
                    \end{bmatrix},
  $$
  therefore
  $$
     \Vert N - N' \Vert_F \leq (\sqrt{n}+\sqrt{m-r}) \Vert Y - Y' \Vert_F + \sqrt{m-r} \Vert A - A' \Vert_F. 
  $$
Thus 
\begin{dmath}
d_{Gr}(\mathcal{H}(A',Y'), \mathcal{H}(A,Y)) = \|P_N - P_{N'}\|_F \leq 2 \|N-N'\|_F \leq  2 (\sqrt{n}+\sqrt{m-r}) (\Vert Y - Y' \Vert_F + \Vert A - A' \Vert_F).
\end{dmath}
\ 
\end{proof}

For small $r$ 
$$(m+n)r - r^2 \ll nm,$$ so to fully utilize the properties of $\wMr$ in computations we first have to find an explicit basis in the horizontal space. This will be done in the next section.

\subsection{Parametrization of the tangent space} \label{notsing:tspar}
To work with low rank matrices it is very convenient to represent them using the truncated singular value decomposition (SVD). Namely for $A \in \Mr$ we have 
$$A = U S V^{\top},$$ 
with $U$ and $V$ having $r$ orthonormal columns and $S$ being a diagonal matrix. Using this notation we find that the following result holds: 
\begin{theorem}\label{notsing:theoremQ}
The orthogonal basis in the kernel of $N$ from \cref{notsing:nullspace2} is given by columns of the following matrix $Q$
$$Q = \begin{bmatrix} V \otimes I_n  & - Y \otimes (U S_1) \\
0 & I_{m-r} \otimes (V S_2)\end{bmatrix},$$
where 

$S_1$ and $S_2$ are diagonal matrices defined as

$$
S_1 =S(S^2 + I_r)^{-\frac{1}{2}}, \quad S_2 = (S^2 + I_r)^{-\frac{1}{2}}
$$

\end{theorem}
\begin{proof}
It suffices to verify that $Q^{\top} Q = I$ and $NQ = 0$ which is performed by direct multiplication.
The number of columns in $Q$ is $nr + (m-r) r$ which is exactly the dimension of the $\mathcal{H}(A,Y)$.
\end{proof}
Now we will use smoothness of $\wMr$ to develop an optimization algorithm over $\Mr$. 
The idea of using kernel of a matrix in optimization problems has appeared before \cite{markovsky2011low, markovsky2013structured}. Algorithm constructed there is a variable--projection–-like method with $O(m^3)$ per iteration complexity, where $m$ is number of columns in the matrix. We explain this approach in more detail in \cref{sec:related}.
\section{Newton method} \label{notsing:secnewton}
\subsection{Basic Newton method} \label{notsing:basicnewton}
Consider the optimization problem 
$$F(A) \rightarrow \min, \quad \mbox{s.t. } A \in \Mr,$$
where $F$ is twice differentiable. Using the idea described in \cref{notsing:intro} this problem is equivalent to 
$$\widehat{F}(A, Y) \rightarrow \min, \quad \mbox{s.t. } (A, Y) \in \wMr,$$
and
$$\widehat{F}(A,Y) = F(A).$$
Following the approach described in e.g. \cite[Section 4.9]{absil2009optimization} we solve this problem by lifting it to the total space $\wMr^{\text{tot}}$ defined by \cref{eq:defining} with the additional condition that the search direction lies in the horizontal bundle $\mathcal{H}$, that is
$$\widetilde{F}(A,Y) \rightarrow \min, \quad \mbox{s.t. } (A, Y) \in \wMr^{\text{tot}}, $$
$$\widetilde{F}(A,Y) = F(A), $$
\begin{equation}\label{eq:small_gauge}
(\delta A, \delta Y) \in \mathcal{H}(A,Y).
\end{equation}

To solve it we will rewrite it using the Lagrange multipliers method, with the 
additional constraint \cref{eq:small_gauge}. Taking into account the defining equations of $\wMr^{\text{tot}}$ \cref{eq:defining} the Lagrangian for the constrained optimization problem reads
$$\mathcal{L}(A, Y, \Lambda, M) = F(A) + \langle AY, \Lambda \rangle + \frac{1}{2} \langle M, Y^{\top} Y - I \rangle,$$ 
where $\Lambda \in \mathbb{R}^{n \times m-r}$ and $M \in \mathbb{R}^{(m-r) \times (m-r)},$ $M^{\top}=M$ are the Lagrange multipliers.
We now find the first-order optimality conditions.
\subsection{First order optimality conditions}
By differentiating $\mathcal{L}$ we find the following equations
$$
\nabla F + \Lambda Y^{\top} = 0, \quad Y M + A^{\top} \Lambda = 0, \quad AY = 0, \quad Y^{\top} Y = I.
$$
Multiplying second equation by $Y^{\top}$ from the left and using equations $AY=0$ and $Y^{\top} Y = I$ we find that $M=0$.
Thus, the first-order optimality conditions reduce to 
\begin{equation}\label{notsing:firstorder}
\nabla F + \Lambda Y^{\top} = 0, \quad  A^{\top} \Lambda = 0, \quad AY = 0, \quad Y^{\top} Y = I.
\end{equation}
\subsection{Newton method and the reduced Hessian system}
Now we can write down the Newton method for the system \cref{notsing:firstorder}, 
which can be written in the saddle point form
\begin{equation}\label{notsing:newton1}
\begin{bmatrix} \widehat{G} & N^{\top} \\
N & 0 
\end{bmatrix}
\begin{bmatrix} \delta z \\
\delta \lambda
\end{bmatrix} 
= \begin{bmatrix} f\\
0
\end{bmatrix} ,
\end{equation}
and
$$f = -\mathrm{vec}(\nabla F + \Lambda Y^{\top}).$$
where we assumed that the initial point satisfies the constraints ($AY = 0, Y^{\top} Y = I_{m-r}$),
the vectors $\delta z$ and $\delta \lambda $ are 
$$
\delta z = \begin{bmatrix} \mathrm{vec}(\delta A) \\ \mathrm{vec}(\delta Y)
\end{bmatrix}, \quad \delta \lambda = 
\begin{bmatrix}
\mathrm{vec}(\delta \Lambda) \\
\mathrm{vec}(\delta M) 
\end{bmatrix},
$$
and the matrix $\widehat{G}$ in turn has a saddle-point structure:
$$
\widehat{G} = \begin{bmatrix}H & C \\
C^{\top} & 0 
\end{bmatrix},
$$
where $H = \nabla^2 F$ is the ordinary Hessian, and $C$ comes from differentiating the term $\Lambda Y^{\top}$ with respect to $Y$ and will be derived later in the text.
The constraints on the search direction $\delta z$ are written as 
$$
N \delta z = 0, 
$$
and 
$$
N = \begin{bmatrix} Y^{\top} \otimes I_n & I_{m-r} \otimes A \\
0 & I_{m-r} \otimes Y^{\top}
\end{bmatrix},
$$
which means that $\delta z$ is in the $\mathcal{H}(A,Y)$ as desired. In what follows our approach is similar to the null space methods described in \cite[Section 6]{benzi2005numerical}.  Using a parametrization via the matrix $Q$ defined in \cref{notsing:theoremQ} we obtain that
$\delta z=Q \delta w$.
\\
The first block row of system \cref{notsing:newton1} reads
$$\widehat{G} Q \delta w + N^{\top} \delta \lambda = f.$$ 
Multiplying by $Q^{\top}$ we can eliminate $\delta \lambda$, which leads to 
the reduced Hessian equation
\begin{equation} \label{reducedHess}
Q^{\top} \widehat{G} Q \delta w = Q^{\top} f.
\end{equation}
Note that $Q^{\top} \widehat{G} Q$ is a small $(n+m)r-r^2 \times (n+m)r-r^2$ matrix as claimed.
We now would like to simplify equation \cref{reducedHess}.
Using the transposition operator defined in \cref{remark}
we find that matrix $C$ is written as
$$C = (I_m \otimes \Lambda) T_{m,m-r}.$$
An important property of the matrix $C$ is that if $Q_{12} = -Y \otimes (US_1)$ is the (1, 2) block of the matrix $Q$, then 
$$Q_{12} C = 0, $$
if $$A^{\top} \Lambda = 0,$$
which is again verified by direct multiplication using the properties of the Kronecker product.
The direct evaluation of the product $$\widehat{G}^{loc} = Q^{\top} \widehat{G} Q,$$ (together with the property above) gives
\begin{equation}\label{notsing:Gdef}
\widehat{G}^{loc} = \begin{bmatrix} Q^{\top}_{11} H Q_{11} & Q^{\top}_{11} H Q_{12} + Q^{\top}_{11} C Q_{22} \\
Q^{\top}_{12} H Q_{11}+ Q^{\top}_{22} C^{\top} Q_{11} & Q^{\top}_{12} H Q_{12}
\end{bmatrix},
\end{equation}
and the system we need to solve has the form
\begin{equation}\label{notsing:mainnewton}
\begin{bmatrix} Q^{\top}_{11} H Q_{11} & Q^{\top}_{11} H Q_{12} + Q^{\top}_{11} C Q_{22} \\
Q^{\top}_{12} H Q_{11}+ Q^{\top}_{22} C^{\top} Q_{11} & Q^{\top}_{12} H Q_{12}
\end{bmatrix}\begin{bmatrix}
\delta u \\
\delta p 
\end{bmatrix}
= \begin{bmatrix}
Q^{\top}_{11} f \\
Q^{\top}_{12} f
\end{bmatrix},
\end{equation}
with 
$$
\quad \delta U \in \mathbb{R}^{n \times r}, \delta P \in \mathbb{R}^{r \times (m - r)}.
$$
We also need to estimate $\Lambda$. 
Recall that to get $Q_{12} C = 0$ we have to require that $A^{\top} \Lambda = 0$ exactly,
thus 
$$\Lambda = Z \Phi, $$
where $Z$ is the orthonormal basis for the left nullspace of $A$, 
and $\Phi$ is defined from the minimization of
$$\Vert \nabla F + Z \Phi Y^{\top} \Vert \rightarrow \min,$$
i.e.
$$\Phi = -Z^{\top} \nabla F Y,$$
and
$$\Lambda = -ZZ^{\top} \nabla F Y.$$
Note that $f$ then is just a standard projection of $\nabla F$ on the tangent space:
$$f = -\mathrm{vec} (\nabla F - ZZ^{\top} \nabla F YY^{\top}) = -\mathrm{vec}(\nabla F - (I - UU^{\top}) \nabla F (I - VV^{\top})),$$
which is always a vectorization of a matrix with a rank not larger than $2r$.
Moreover, 
\begin{equation}\label{notsing:g1def}
g_1 = Q^{\top}_{11} f = (V^{\top} \otimes I) f = 
\end{equation}
$$
- \mathrm{vec}((\nabla F - (I - UU^{\top}) \nabla F (I - VV^{\top})) V) = \mathrm{vec}(-\nabla F V),
$$
and the second component
\begin{equation}\label{notsing:g2def}
g_2 = Q^{\top}_{12} f = -(Y^{\top} \otimes (US_1)^{\top}) f =
\end{equation}
$$  \mathrm{vec}((U S_1)^{\top}(\nabla F - (I - UU^{\top}) \nabla F (I - VV^{\top}))Y) =  \mathrm{vec}( S^{\top}_1 U^{\top}\nabla F Y).$$
The solution is recovered from $\delta u$, $\delta p$ as
$$\delta a = (V \otimes I_n) \delta u - (Y \otimes (US_1)) \delta p,$$
or in the matrix form,
$$\delta A = \delta U V^{\top}  - U S_1 \delta P Y^{\top},$$
and the error in $A$ (which we are interested in) is given by 
$$\Vert \delta A \Vert_F^2 = \Vert \delta U \Vert_F^2 + \Vert S_1 \delta P \Vert_F^2.$$
We can further simplify the off-diagonal block. Consider
$$\widehat{C} = Q^{\top}_{11} C Q_{22} = (V^{\top} \otimes I) (I \otimes \Lambda) T (I \otimes V) (I \otimes S_2).$$
Then multiplication of this matrix by a vector takes the form:
$$ \mathrm{mat}(\widehat{C} \mathrm{vec}(\Phi))  = \Lambda (V S_2 \Phi)^{\top} V = \Lambda \Phi^{\top} S^{\top}_2 V^{\top} V = \Lambda \Phi^{\top} S^{\top}_2,$$
thus
$$
\widehat{C} = (S_2 \otimes \Lambda) T_{r, n-r}.
$$
\subsection{Retraction}\label{sec:retraction}
Note that since we assumed that the initial points satisfy the constraints
\begin{equation} \label{notsing:deseqs}
AY = 0,\quad Y^{\top} Y = I_{m-r},
\end{equation}
after doing each step of the Newton algorithm we have to perform the retraction back to the manifold $\wMr^{tot}$. 
One such possible retraction is the following. Define a map 
$$R: \wMr^{tot}  \oplus \mathcal{H} \to \wMr^{tot},$$
$$R((A,Y),(\delta Y, \delta A)) = (R_1(A, \delta A), R_2(Y, \delta Y))$$
$$R_2(Y,\delta Y) = \textbf{qf}(Y + \delta Y) = Y_1,$$
$$R_1(A,\delta A) = A(I - Y_1 Y_1^{\top})$$
where $\textbf{qf}(\xi)$ denotes the Q factor of the QR-decomposition of $\xi$, which is a standard second-order retraction on the Stiefel manifold \cite[Example 4.1.3]{absil2009optimization}.

In the fast version of the Newton method which will be derived later, we will also use the standard SVD-based retraction which acts on the matrix $A+\delta A$ simply by truncating it's SVD to the rank $r$. It is also known that given the SVD of the matrix $A$ then for certain small corrections $\delta A$, the SVD of $A+\delta A$ can be recomputed with low computational cost as described in \cite[\S 3
]{vandereycken2013low}. It is also known to be a second order retraction \cite{absil2012projection}. We denote this operation by $R_{\text{SVD}}(A, \delta A)$.
\subsection{Basic Algorithm}
The basic Newton method on the manifold $\wMr$ is summarized in the following algorithm
\\
\begin{algorithm}[H]
	\caption{Newton method} \label{notsing:alg1}
	\begin{algorithmic}[1]
  \STATE{Initial conditions $A_0,Y_0$, functional $F(A)$} and tolerance $\varepsilon$
  \\
  \STATE{Result: minimum of $F$ on $\Mr$}	
	\WHILE{$\Vert \delta U^i \Vert^2_F + \Vert (S_1)^i \delta P^i \Vert^2_F>\varepsilon$}
	
	\STATE{	
		$U^i,S^i,V^i$ = svd($A^i$)}
		\\
	\STATE{Solve 
		$\widehat{G}^{loc}
		\begin{bmatrix}
		\delta u^i \\
		\delta p^i
		
		\end{bmatrix}
		=
		\begin{bmatrix}
		g_1 \\
		g_2
				
		\end{bmatrix}
		$ where $\widehat{G}^{loc}, g_1, g_2 $ are defined by formulas \cref{notsing:Gdef,notsing:g1def,notsing:g2def}}
		\\
		
		\STATE{
		$\delta A^i = \delta U^i V^{i \top} - U^i (S_1)^i \delta P^i Y^{i \top}$
		\\
		$\delta Y^i = V^i (S_2)^i \delta P^i$}
		\\
		\STATE{
		$A^{i+1},Y^{i+1}  = R((A^i,\delta A^i),(Y^i,\delta Y^i))$}
		\\
		\STATE{
		$i = i +1$ }
	
\ENDWHILE
\RETURN{$A^i$}
\end{algorithmic}
\end{algorithm}
Even though this algorithm demonstrates that our approach is rather inefficient in terms of memory and complexity -- storing and doing multiplications by $Y$ are of order $O(m^2)$ instead of desired $O((n+m)r)$. We resolve this issue in the next section. Analysis of the convergence and behavior of the algorithm near the points $(A,Y)$ corresponding to matrices of strictly smaller rank is performed in \cref{sec:behavior}.
\subsection{Semi-implicit parametrization of the tangent space}
\label{sec:trick}
Let us introduce a new variable 
$$\delta \Phi^{\top} = Y \delta P^{\top}, $$
$$\delta \Phi \in \mathbb{R}^{r \times m}.$$
This results in an implicit constraint on $\delta \Phi$ 
$$ \delta \Phi V= 0.$$
In order to make an arbitrary $\Phi$ satisfy it, we first multiply it by the projection operator $I - VV^{\top},$
$$ \Phi' = \Phi (I-VV^{\top}),$$ or in the matrix form
$$\begin{bmatrix}
\delta u \\
\delta \phi'
\end{bmatrix} =
\begin{bmatrix}
I & 0 \\
0 & I-VV^{\top} \otimes I
\end{bmatrix}
\begin{bmatrix}
\delta u \\
\delta \phi
\end{bmatrix}.
$$
Notice also that  $$\delta P = \delta \Phi Y,$$ and again using the properties of the Kronecker product we obtain
$$
\begin{bmatrix}
\delta u \\
\delta p
\end{bmatrix} =
\begin{bmatrix}
I & 0 \\
0 & Y^{\top}\otimes I
\end{bmatrix}
\begin{bmatrix}
\delta u \\
\delta \phi
\end{bmatrix}.
$$
Denote
$$
\Pi = 
\begin{bmatrix}
I & 0 \\
0 & I-VV^{\top} \otimes I
\end{bmatrix},
$$
$$
W = 
\begin{bmatrix}
I & 0 \\
0 & Y^{\top}\otimes I
\end{bmatrix}.
$$
The equations for the Newton method in the new variables take the following form:
\begin{equation} \label{notsing:advnewton}
\Pi^{\top}W^{\top} \widehat{G}^{loc} W\Pi 
\begin{bmatrix}
	\delta u \\
	\delta \phi
\end{bmatrix} =
\Pi^{\top} W^{\top}
\begin{bmatrix}
	g_1 \\
	g_2
\end{bmatrix},
\end{equation}
where $g_1,g_2, \widehat{G}^{loc}$ are as in \cref{notsing:g1def,notsing:g2def,notsing:Gdef} and the linear system in \cref{notsing:advnewton} is of size $(n+m)r$.
\\
\subsection{Iterative method}
For a large $n$ and $m$ forming the full matrix \cref{notsing:advnewton} is computationally expensive, so we switch to iterative methods.
To implement the matvec operation we need to simplify
$$ \begin{bmatrix}
l_1 \\
l_2 
\end{bmatrix}=\Pi^{\top} W^{\top}  \widehat{Q}^{loc} W \Pi
\begin{bmatrix}
\delta u \\
\delta \phi 
\end{bmatrix},
$$ first.
\\
Direct computation shows that
$$
\begin{bmatrix}
l_1  \\
l_2 
\end{bmatrix} = 
\Pi^{\top} W^{\top}
\left[ 
\begin{array}{c}
(V^{\top} \otimes I) H (V \otimes I) \delta u -
\\ (V^{\top} \otimes I)H(I \otimes U ) \mathrm{vec}(S_1 \delta \Phi  (I - VV^{\top}))-\\
\mathrm{vec}((I-UU^{\top})\nabla F (I-V V^{\top}) \delta \Phi^{\top} S_2)
\\
\\
-(Y^{\top} \otimes I) (I \otimes S_1) (I \otimes U^{\top}) H (V \otimes I) \delta u +\\
(Y^{\top} \otimes I) \mathrm{vec}(S_2 (\delta U)^{\top} (-(I-U U^{\top}) \nabla F)) +
\\ (Y^{\top} \otimes I)(I \otimes S_1) (I\otimes U^{\top} ) H (I \otimes U) \mathrm{vec} (S_1 \delta \Phi (I-VV^{\top}))
\end{array}
\right],
$$
and the right hand side has the following form:
$$\begin{bmatrix}
g_1^{\prime} \\
g_2^{\prime}
\end{bmatrix} = 
\Pi^{\top} W^{\top}
\begin{bmatrix}
-\mathrm{vec} \nabla F V \\
(Y^{\top} \otimes I) \mathrm{vec} (S_1 U^{\top} \nabla F)
\end{bmatrix}.
$$
Since both $\Pi$ and $W$ only act on the second block it is easy to derive the final formulas:
\begin{dmath}
\label{eq:mv1}
l_1 = (V^{\top} \otimes I) H (V \otimes I) \delta u - (V^{\top} \otimes I)H(I \otimes U ) \mathrm{vec}(S_1 \delta \Phi  (I - VV^{\top}))-\mathrm{vec}((I-UU^{\top})\nabla F (I-V V^{\top}) \delta \Phi^{\top} S_2),
\end{dmath}
\begin{dmath}
\label{eq:mv2}
l_2  = 
-(I - VV^{\top} \otimes I) (I \otimes S_1) (I \otimes U^{\top}) H (V \otimes I) \delta u + \mathrm{vec}(S_2 (\delta U)^{\top} (-(I-U U^{\top}) \nabla F (I - V V^{\top})) +(I-VV^{\top}\otimes I)(I \otimes S_1) (I\otimes U^{\top} ) H (I \otimes U) \mathrm{vec} (S_1 \delta \Phi(I-VV^{\top})),
\end{dmath}
\begin{equation}\label{eq:g1_mv}
g_1^{\prime} = -\mathrm{vec} \nabla F V,
\end{equation}
\begin{equation}\label{eq:g2_mv}
g_2^{\prime} = \mathrm{vec} (S_1 U^{\top} \nabla F (I-VV^{\top})).
\end{equation}
Note that in new variables we obtain
$$\delta A = \delta U V^{\top} - U S_1 \delta \Phi,$$
and 
$$A+\delta A = U(SV^{\top} - S_1 \delta \Phi) + \delta U V^{\top}.$$
Using this representation of $A+\delta A$ we can recompute its SVD without forming the full matrix as described in \cref{sec:retraction}. 
This allows us not to store the matrix $A$ itself but only the $U$, $S$ and $V$ that we get from the SVD.
We obtain the following algorithm
\begin{algorithm}[H] 
	\caption{Fast Newton method} \label{notsing:alg2}
	\begin{algorithmic}[1]
		\STATE{Initial conditions $U_0,S_0,V_0$, functional $F(A)$} and tolerance $\varepsilon$
		\\
		
		\STATE{Result: minimum of $F$ on $\Mr$}
		
		\WHILE{$\Vert \delta U^i \Vert^2 + \Vert (S_1)^i \delta \Phi^i \Vert^2_F>\varepsilon$}
		
		\STATE{
		Solve linear system with matvec defined by formulas (\ref{eq:mv1}),(\ref{eq:mv2}) and right hand side defined by formulas \cref{eq:g1_mv,eq:g2_mv} using GMRES, obtaining $\delta u^i, \delta \phi^i.$}
				\\
				\STATE{
			$\delta A^i = \delta U^i V^{i \top} - U^i (S_1)^i \delta \Phi^i$}
			\\
			\STATE{
			$U^{i+1},S^{i+1},V^{i+1} = R_{\text{SVD}}(A^i,\delta A^i)$}
			\\
			\STATE{
			$i = i +1$
}	
		\ENDWHILE
		\RETURN{$U^i,S^i,V^i$}
	\end{algorithmic}
\end{algorithm}
\pagebreak
\section{Technical aspects of the implementation}\label{notsing:techstuff}
\subsection{Computation of the matvec and complexity}
To efficiently compute the matvec for a given functional $F$ one has to be able to evaluate the following expressions of the first order:
\begin{equation} \label{notsing:first-order}
\nabla F V , (\nabla F)^{\top} U ,\nabla F \delta X,\delta X \nabla F, 
\end{equation}
and of the second order:
\begin{equation} \label{notsing:second-order}
(V^{\top} \otimes I) H (V \otimes I) \delta x, (V^{\top} \otimes I) H (I \otimes U) \delta x
\end{equation}
$$(I\otimes U^{\top}) H (V \otimes I) \delta x,(I \otimes U^{\top}) H (I \otimes U) \delta x.$$
The computational complexity of \cref{notsing:alg2} depends heavily on whether we can effectively evaluate \cref{notsing:first-order,notsing:second-order}, which, however, any similar algorithm requires.
Let us now consider two examples.
\subsection{Matrix completion} \label{notsing:mc}
Given some matrix $B$ and a set of indices $\Gamma$ define 
$$F(x) = \frac{1}{2}\sum_{(i,j) \in \Gamma} (x_{i,j} - B_{i,j})^2 \to \min, x \in \Mr.$$
Then $$\nabla F_{ij} = x_{ij} - B_{ij}, (i,j) \in \Gamma,$$
$$\nabla F_{ij} = 0, (i,j) \notin \Gamma.$$
Then $H$ in this case is a diagonal matrix with ones and zeroes on the diagonal, the exact position of which are determined by $\Gamma$. Assuming that the cardinality of $\Gamma$ is small, the matrix products from \cref{notsing:first-order} can be performed efficiently by doing sparse matrix multiplication. Note that multiplication by $H$ in \cref{notsing:second-order} acts as a mask, turning the first factor into a sparse matrix, allowing for effective multiplication by the second factor.
\subsection{Approximation of a sparse matrix} \label{notsing:ma}
Consider the approximation functional
$$F(x) = \frac{1}{2} \| x - B \|_F^2 \to \min, x \in \Mr,$$
and $B$ is a sparse matrix.
Then
$$\nabla F = x-B,$$ and 
expressions \cref{notsing:second-order}, can be heavily simplified by noticing that $H$ in this case is the identity matrix and the sparseness of $B$ is used to evaluate \cref{notsing:first-order}.
\section{Numerical results}\label{notsing:numerical}
\subsection{Convergence analysis}
\cref{notsing:alg2} was implemented in Python using \texttt{numpy} and \texttt{scipy} libraries.
We tested it on the functional described in \cref{notsing:ma} for $B$ being the matrix constructed from the MovieLens 100K Dataset \cite{harper2016movielens}, so $n=1000,m=1700$ and $B$ has $100000$ non-zero elements.
Since the pure Newton method is only local, for a first test we choose small random perturbation (in the form $0.1 \mathcal{N} (0,1)$) of the solution obtained via SVD as initial condition. We got the following results for various $r$ (see \cref{fig:ma}).
\begin{figure}[htp]
	\centering
	\begin{subfigure}{.45\textwidth}
		\centering
		     	\includegraphics[width=\linewidth]{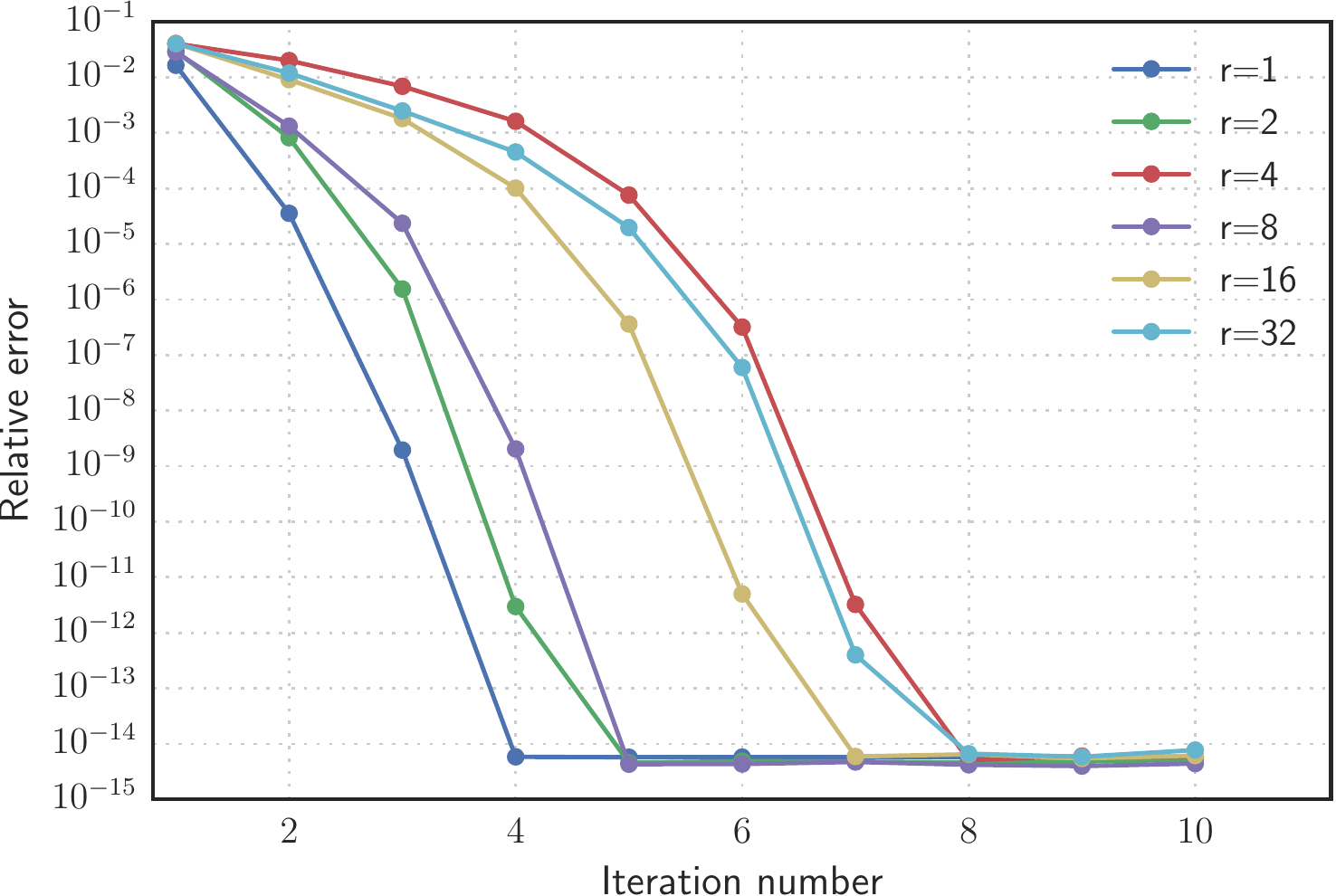}
		     	\caption{Close initial guess. }
		     	\label{fig:ma}
	\end{subfigure}%
\begin{subfigure}{.45\textwidth}
	\centering
	\begin{tabular}{cccc}
		\includegraphics[width =0.45\linewidth]{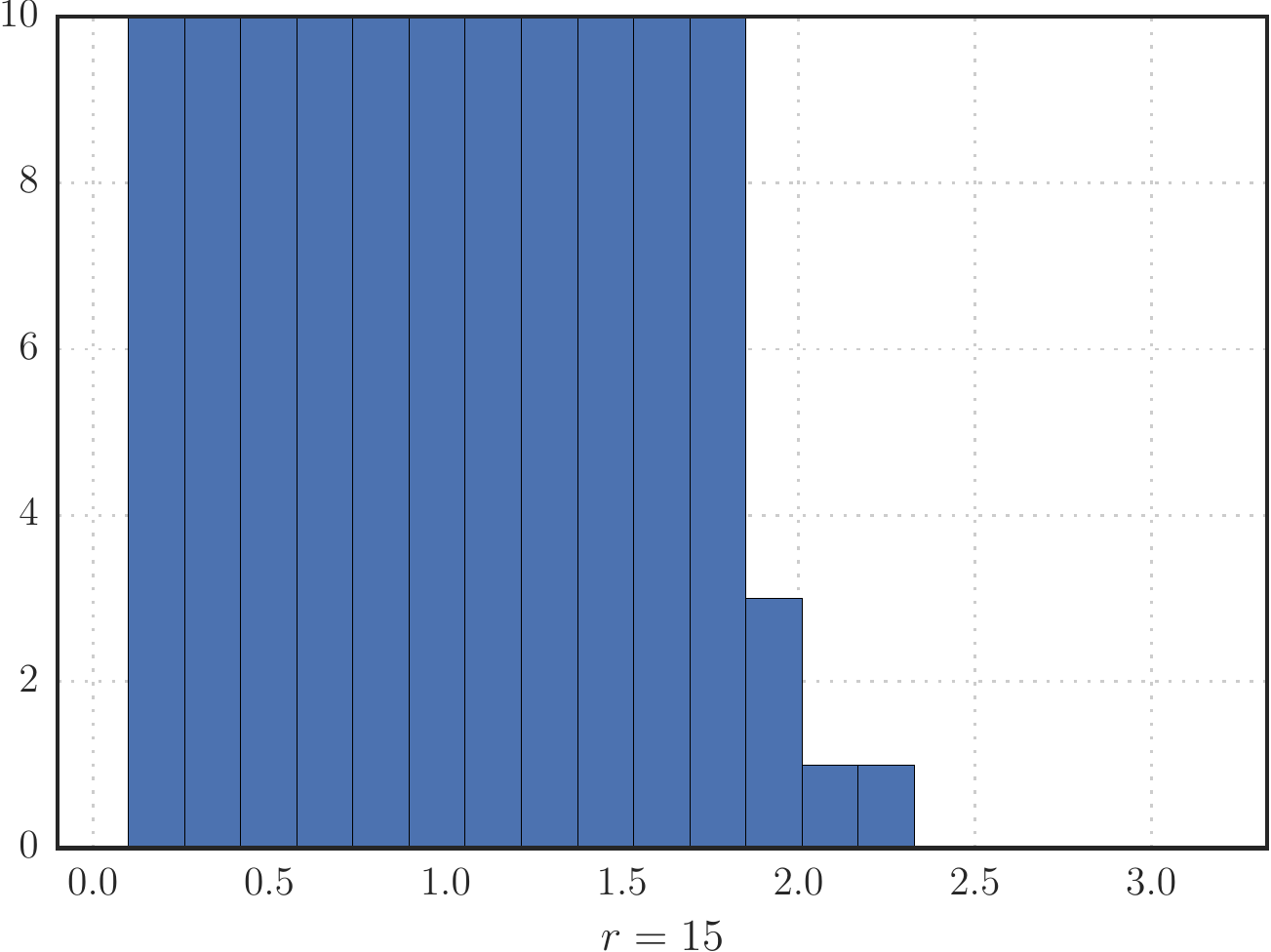} &
		\includegraphics[width =0.45\linewidth]{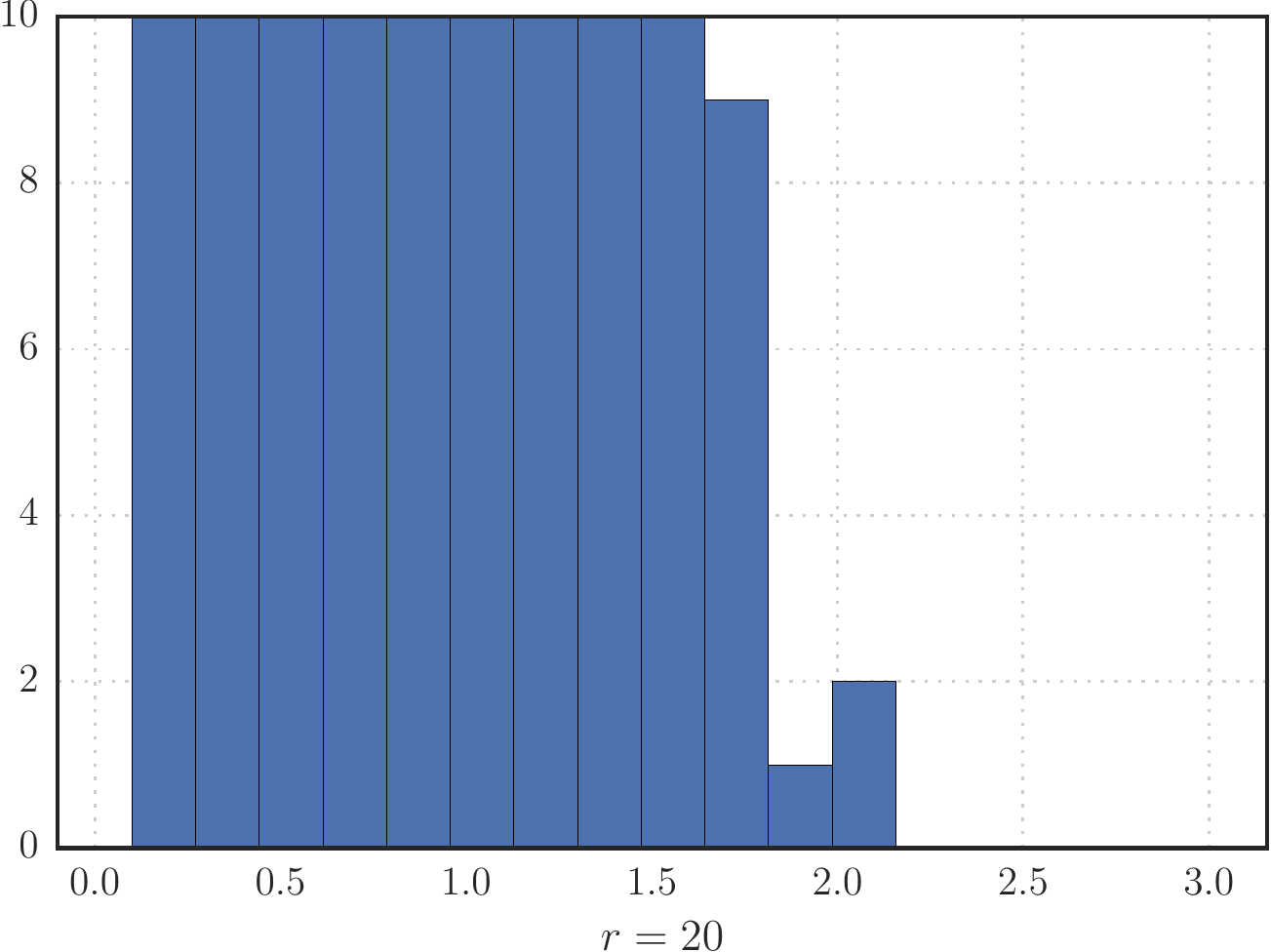} \\
		\includegraphics[width =0.45\linewidth]{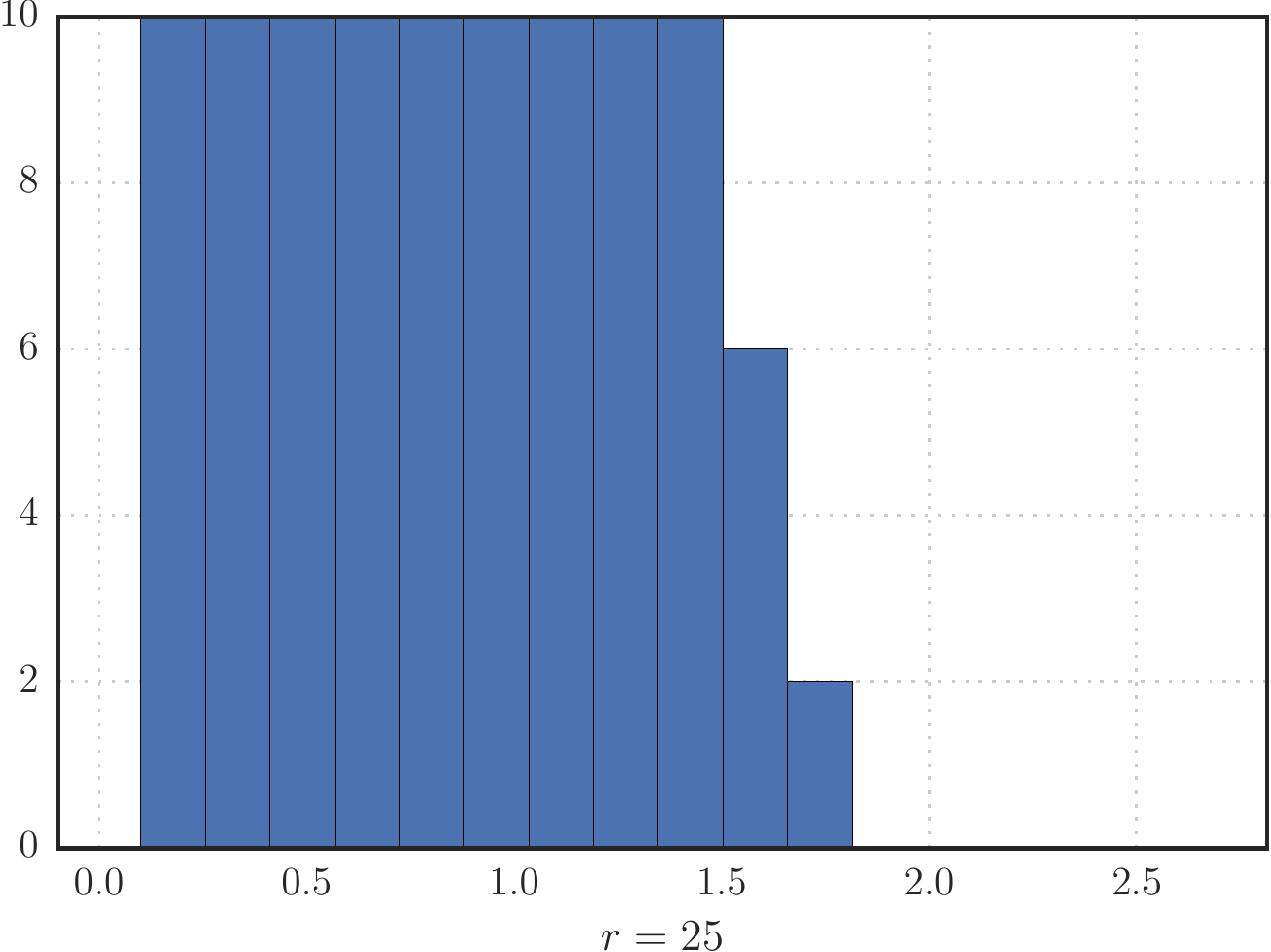} &
		\includegraphics[width =0.45\linewidth]{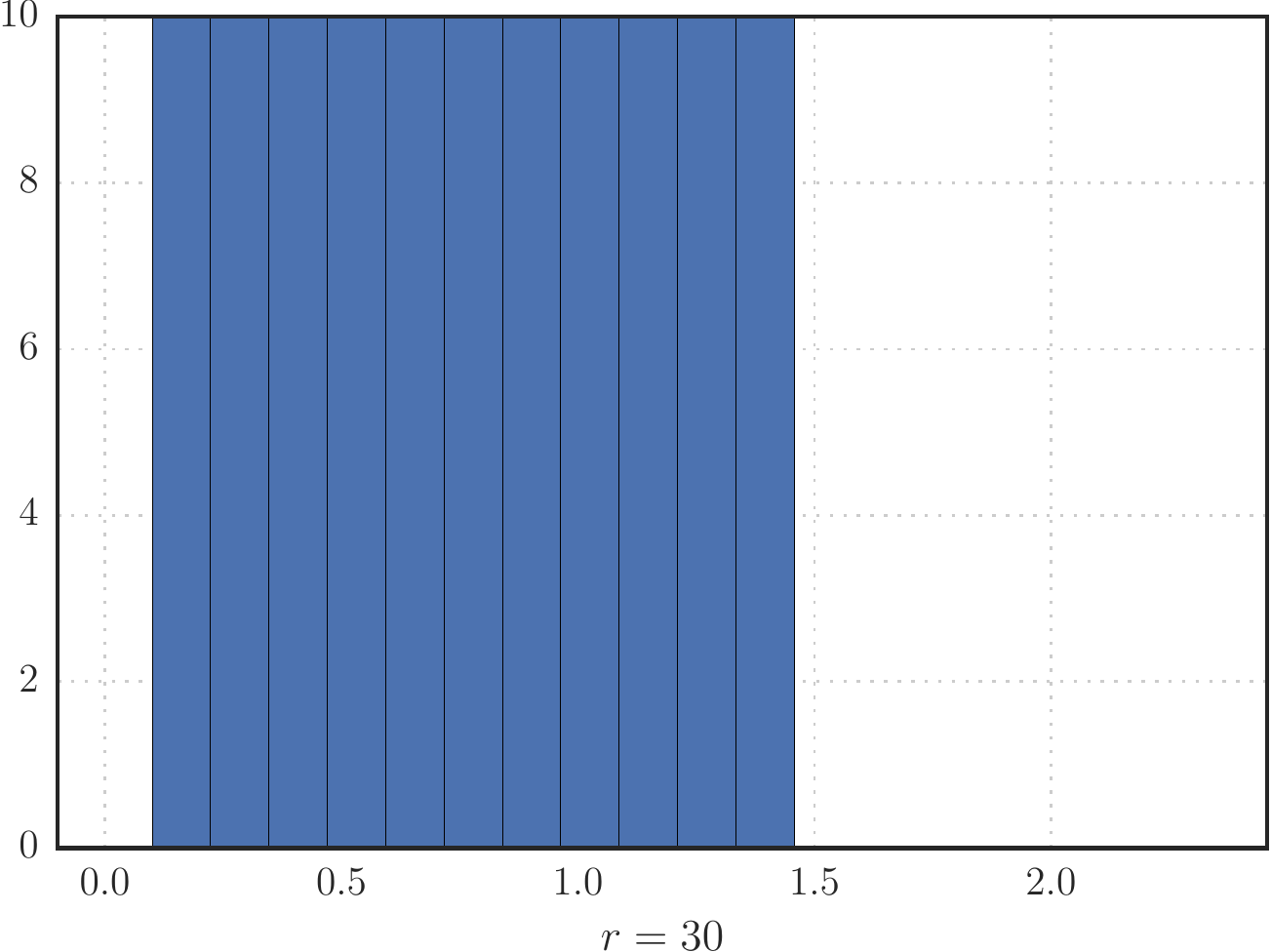} 
	\end{tabular}
	\caption{Convergence histograms.}
	\label{fig:hist}
	\end{subfigure}
	\caption{Sparse matrix approximation: test of local convergence.}
	\end{figure}
This shows the quadratic convergence of our method.

Now we fix the rank and test whether the method converges to the exact answer for a perturbation of the form $\alpha \mathcal{N}(0,1)$ for various $\alpha$ and plot a number of convergent iterations vs $\alpha \in [0.1,2.5]$ (see \cref{fig:hist}).
We see that for a sufficiently distant initial condition the method does not converge to the desired answer. To fix this we introduce a simple version of the trust-region algorithms described in \cite{yuan2000review} (to produce initial condition we perform a few steps of the power method). Results are summarized in \cref{diffr-tr}.
\begin{figure}[htp]
	\centering
	\begin{tabular}{cccc}
		\subcaptionbox{$r=20$\label{r20-2}}{\includegraphics[width =0.45\linewidth]{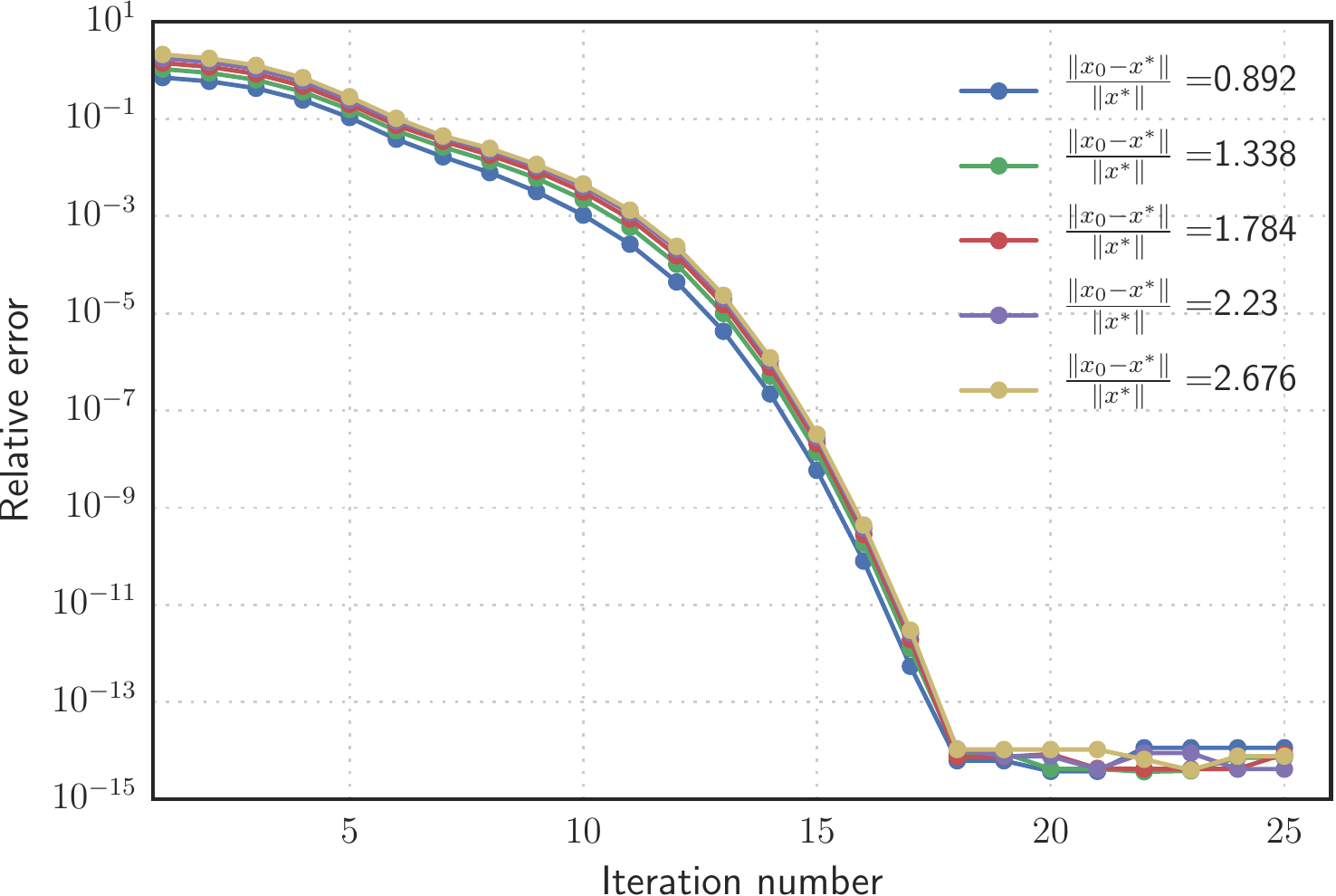}} &
		\subcaptionbox{$r=25$\label{r25-2}}{\includegraphics[width = 0.45\linewidth]{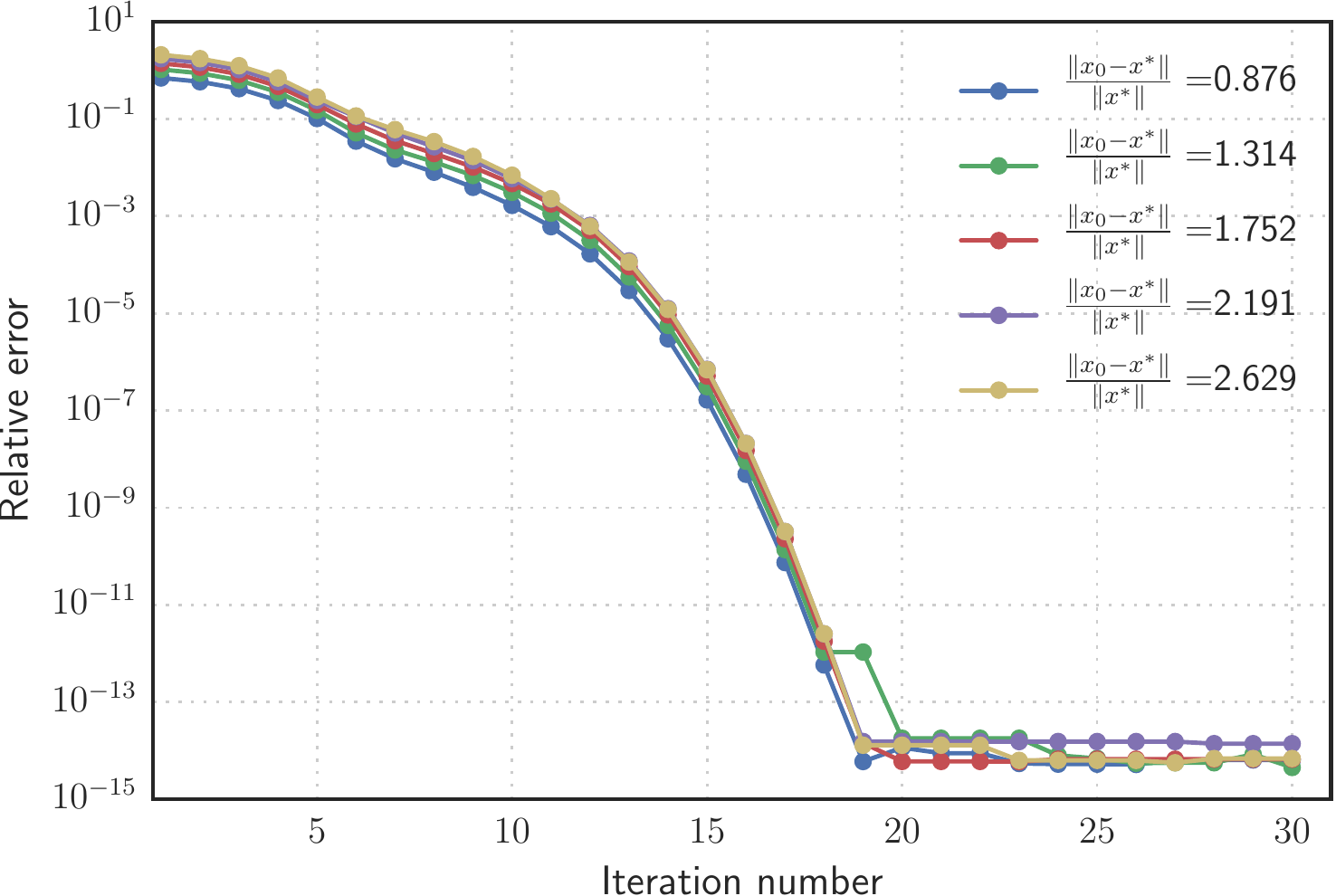}} \\
		\subcaptionbox{$r=30$\label{r30-2}}{\includegraphics[width = 0.45\linewidth]{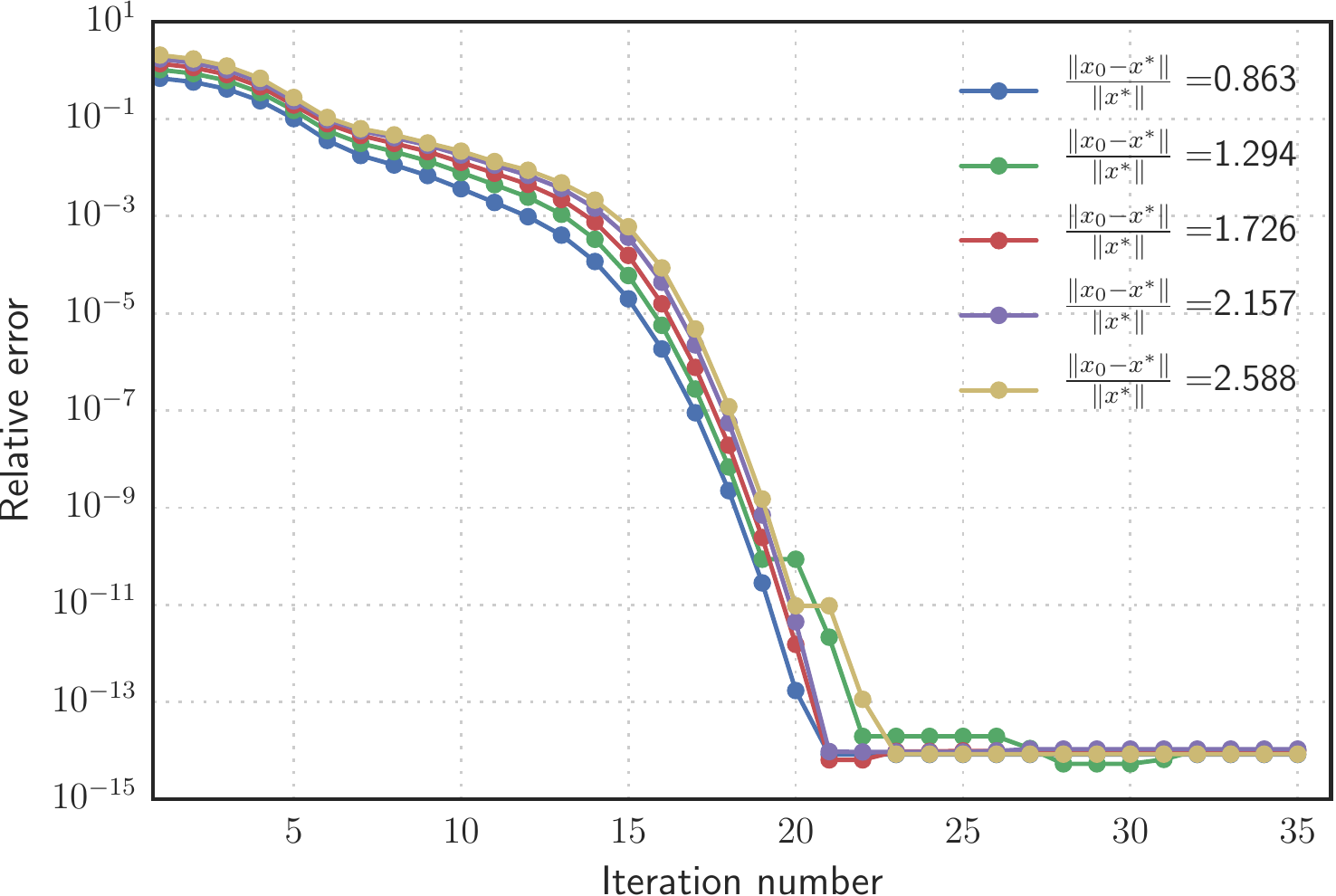}} &
		\subcaptionbox{$r=35$\label{r35-2}}{\includegraphics[width = 0.45\linewidth]{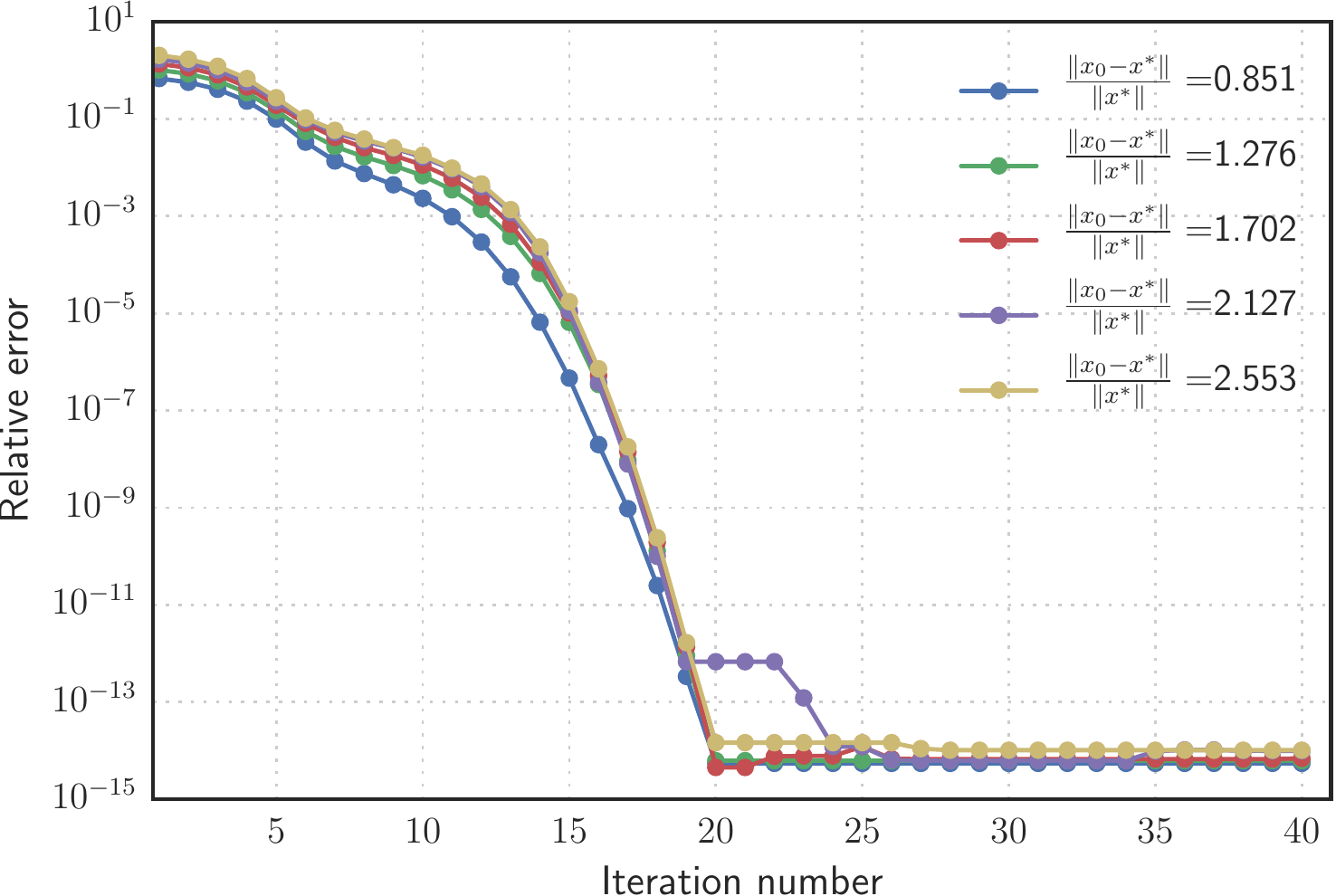}}\\
	\end{tabular}
	\caption{Sparse matrix approximation: trust-region method.}
	\label{diffr-tr}
\end{figure}
We also test our algorithm for the matrix completion problem. As an initial data we choose first $15000$ entries in the database described above. Using the same trust-region algorithm we obtained the following results (see \cref{fig:mc}).

As a final test we show quadratic convergence even in a case when the exact solution is of rank smaller than $r$ for which the method is constructed. To do this we take first $k$ elements of the dataset for various $k$, find the rank $r_0$ of the matrix constructed from these elements, and run the trust-region Newton method for $r=r_0+10$. The results are presented in \cref{fig:mc-final}. 
\begin{figure}[htp]
	\centering
	\begin{tabular}{cc}
		\subcaptionbox{Trust region method.\label{fig:mc}}{\includegraphics[width =0.45\linewidth]{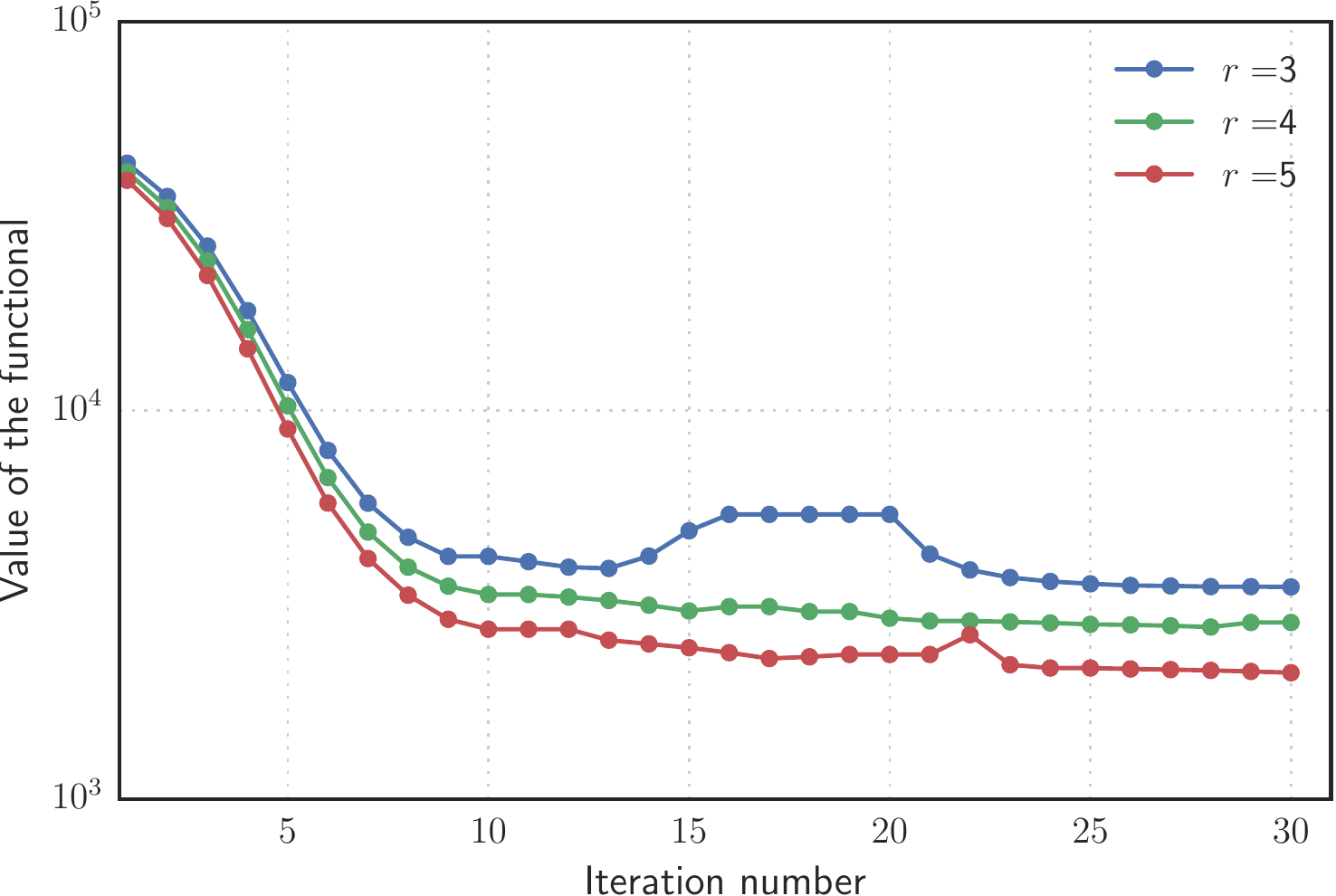}} &
		\subcaptionbox{Quadratic convergence in the case of rank defficiency.\label{fig:mc-final}}{\includegraphics[width = 0.45\linewidth]{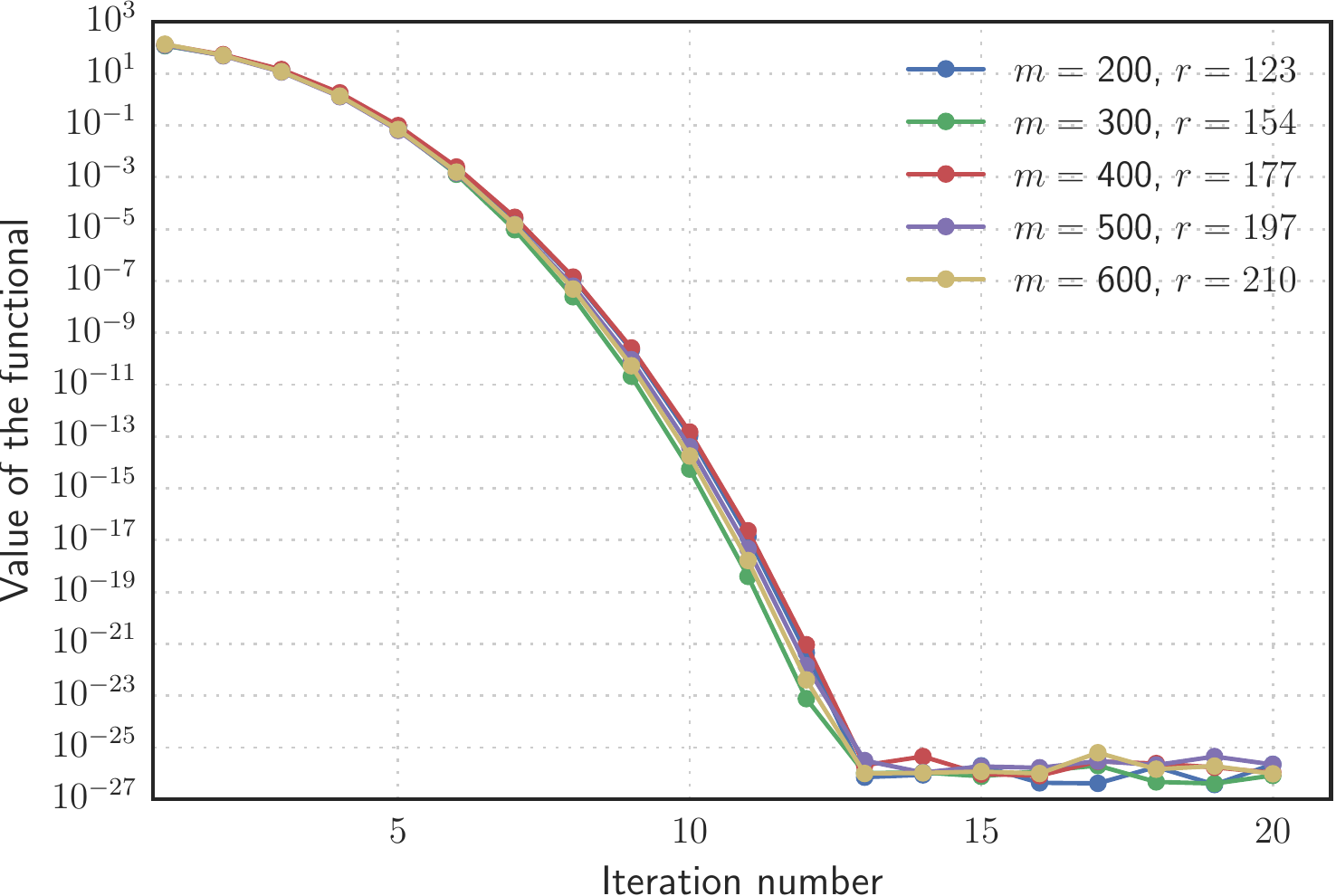}} \\
	\end{tabular}
	\caption{Matrix completion tests.}
\end{figure}
\subsection{Behavior of the algorithm in cases of rank deficiency and underestimation}
\label{sec:behavior}
In the Newton equation of \cref{notsing:alg1}, one has to solve a linear system with 
$$\widehat{G}^{loc}= \begin{bmatrix}
Q^{\top}_{11} H Q_{11} & Q^{\top}_{11} H Q_{12} + Q^{\top}_{11} C Q_{22} \\
 Q^{\top}_{12} H Q_{11} + Q^{\top}_{22} C Q_{11}& Q^{\top}_{12} H Q_{12}
\end{bmatrix},
$$
with $H$ the Hessian of the objective function $F : \mathbb{R}^{n \times m} \to \mathbb{R}$, which we can assume to be positive definite. Suppose that a matrix of rank $<r$ is the global minimum of $F$. Then $S_1$ is singular and $\Lambda  = 0$, which in turn imply that $Q_{12} = -Y \otimes  (U S_1)$ is singular and $C = 0$. Hence, the matrix $\widehat{G}^{loc}$ is singular. It is easy to understand the reason of this behavior. The function $\widehat{F}$ defined on $\wMr$ now has non-unique critical point,  --- the set of critical points is now in fact a submanifold of $\wMr$. Thus any vector tangent to this submanifold will be a solution of the Newton system. An analysis of the behavior of the Newton method for such functions is studied in e.g. \cite{decker1980newton}. While we plan to analyze it and prove quadratic convergence in our future work, now we note that Krylov iterative methods handle singular systems if we choose initial condition to be the zero vector, and quadratic convergence has been observed in all the numerical experiments.
 
We will now compare our method (desN) with the reduced Riemannian Newton (rRN) (which is also known as constrained Gauss-Newton method \cite{kressner2016preconditioned}) and CG methods on the fixed-rank matrix manifolds for the approximation problem. The former is obtained by neglecting the curvature term involving $S^{-1}$ in the Hessian (see \cite[Proposition 2.3]{vandereycken2013low}) and for the latter we use the \texttt{Pymanopt} \cite{JMLR:v17:16-177} implementation. We choose $n=m=30, r=10$ and for the first test we compare the behavior of these algorithms in the case of the exact solution being of rank $r_0 < r$ with $r_0 = 5$. In the second test, we study the converse situation when the rank is underestimated --- the exact solution has rank $r_0 > r$ with $r_0=15$. As before, for the reference solution we choose a truncated SVD of the approximated matrix. The results are summarized in the \cref{fig:comp-good,fig:comp-bad}. Note that the case of rank underestimation was also studied in \cref{fig:mc-final}. We observe that the proposed algorithm maintains quadratic convergence in both cases. Even though the reduced Riemannian Newton method is quadratic in the case of rank deficiency, it becomes linear in the case of rank underestimation. This phenomenon is well-known and explained e.g. in \cite[Section 5.3]{kressner2016preconditioned} and is related to the fact that when exact minimum is on the variety this approximate model in fact becomes exact second order model. CG is linear in both cases.
\begin{figure}[htp]
\label{fig:compare-conv}
	\centering
	\begin{tabular}{cc}
		\subcaptionbox{Rank defficiency case \label{fig:comp-good}}{\includegraphics[width =0.45\linewidth]{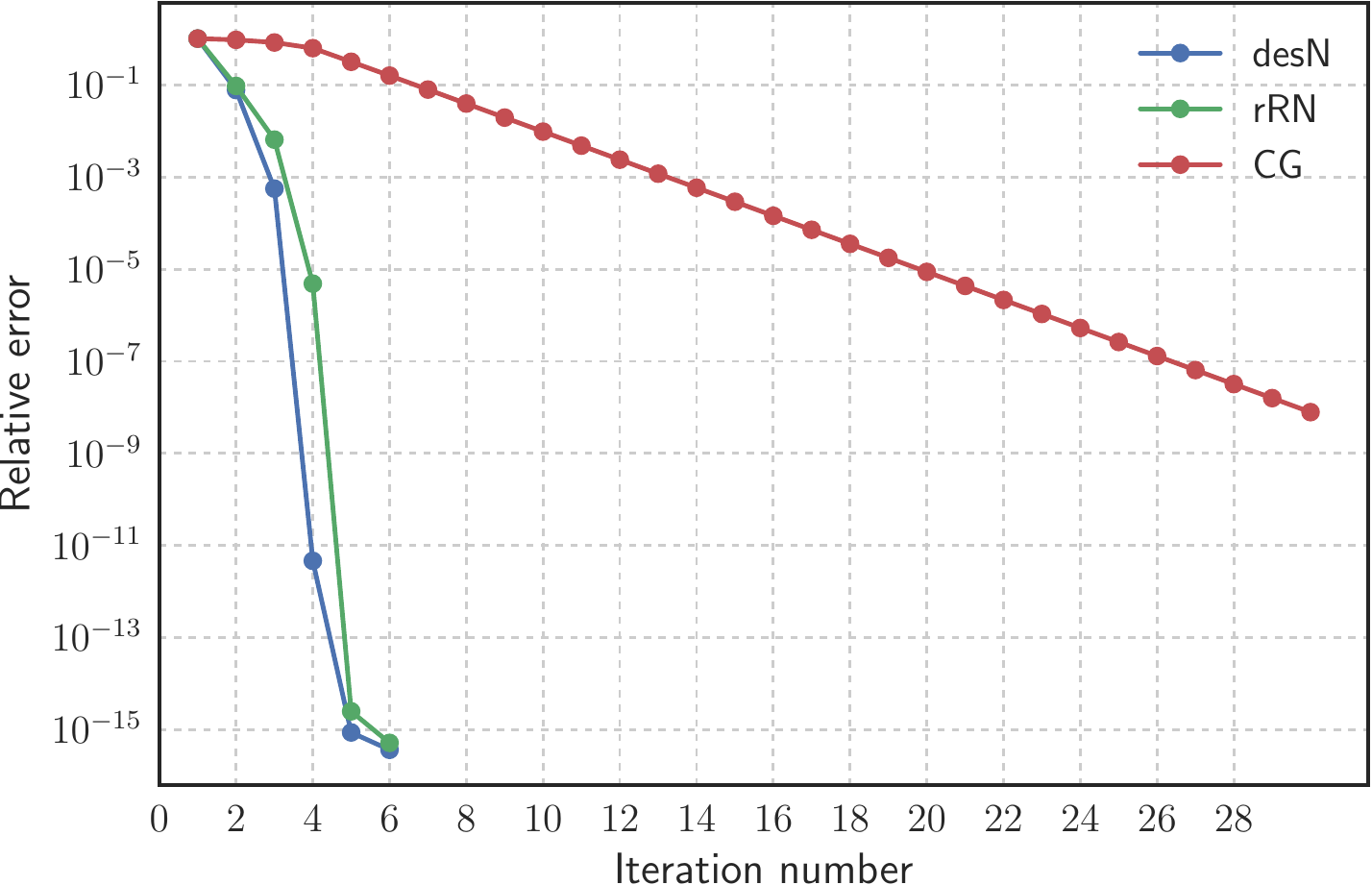}} &
		\subcaptionbox{Rank underestimation.\label{fig:comp-bad}}{\includegraphics[width = 0.45\linewidth]{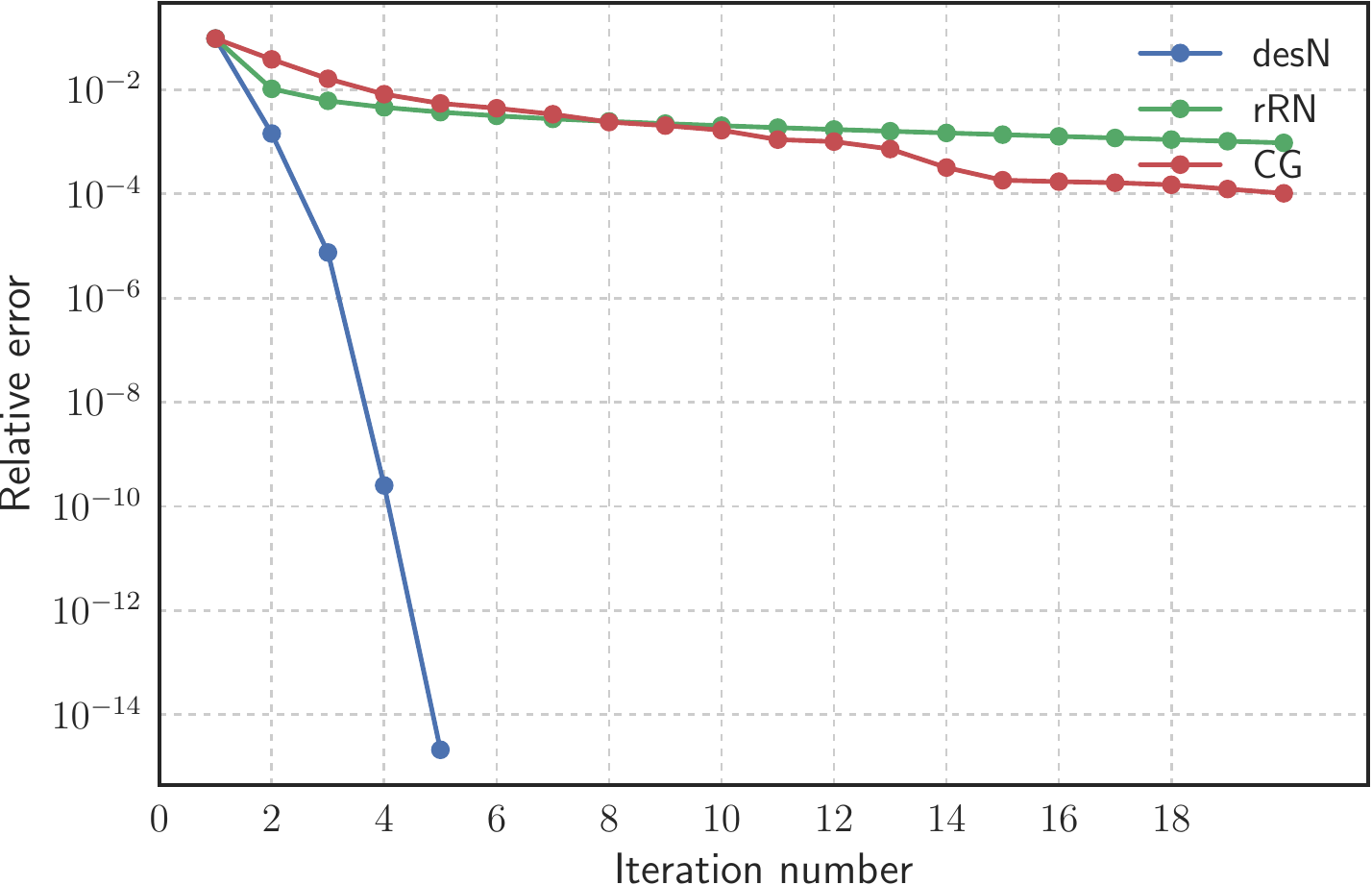}} \\
	\end{tabular}
	\caption{Comparison of the convergence behaviour of various optimization algorithms.}
\end{figure}
\subsection{Comparison with the regularized Newton method}

In this subsection we will compare behavior of our method and of the full Riemannian Newton method on the low-rank matrix variety. To avoid problems with zero or very small singular values we choose some small parameter $\varepsilon$, and in the summands involving $S^{-1}$ in the formulas for the Hessian matrix \cite[Proposition 2.3]{vandereycken2013low} we use the regularized singular values
$$\sigma_i^{\varepsilon} = \max \lbrace \sigma_i, \varepsilon \rbrace,$$ thus obtaining regularized Newton method (regN).
As a test problem we choose a matrix completion problem where the exact answer is known (given sufficiently many elements in the matrix) and of a small rank. To construct such a matrix $A$ we take the uniform grid of size $N=40$ in the square $[-1, 1]^2$ and sample values of the function
$$f(x, y) = e^{-x^2-y^2},$$
on this grid. It is easy to check that this matrix has rank exactly $1$. We choose $r_0 = 5$ and compare relative error with respect to the exact solution $A$, value of the functional as defined in \cref{notsing:mc} and value of the second singular value $\sigma_2$. 
\begin{figure}[htp]
	\centering
	\begin{tabular}{lr}
		\subcaptionbox{Relative error w.r.t the exact answer \label{illcompl-relerr}}{\includegraphics[width =0.45\linewidth]{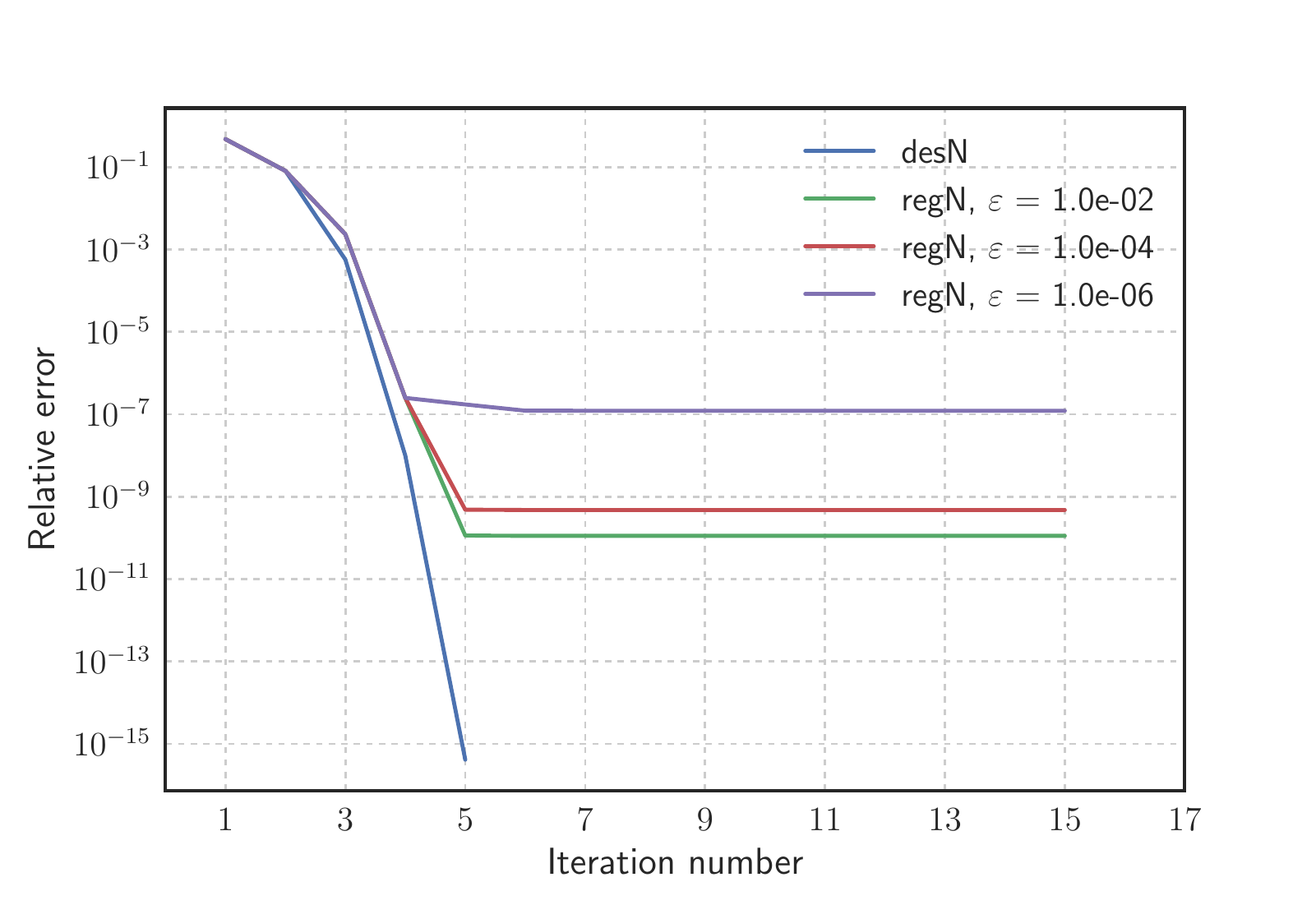}} &
		\subcaptionbox{$\sigma_2$ of the current iterate \label{illcompl-sig2}}{\includegraphics[width = 0.45\linewidth]{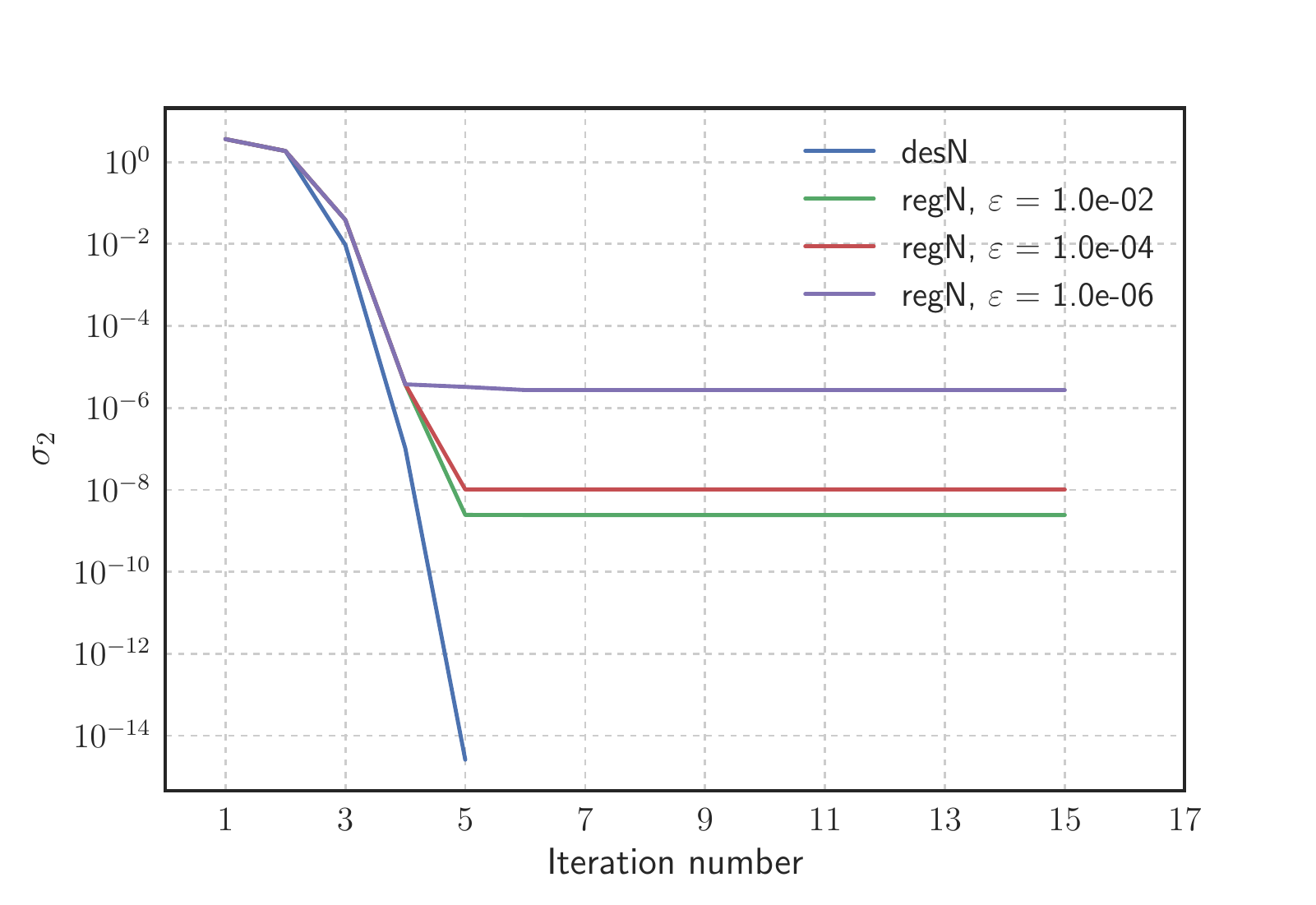}} 
	\end{tabular}
		\subcaptionbox{Value of the functional \label{illcompl-func}}{\includegraphics[width = 0.45\linewidth]{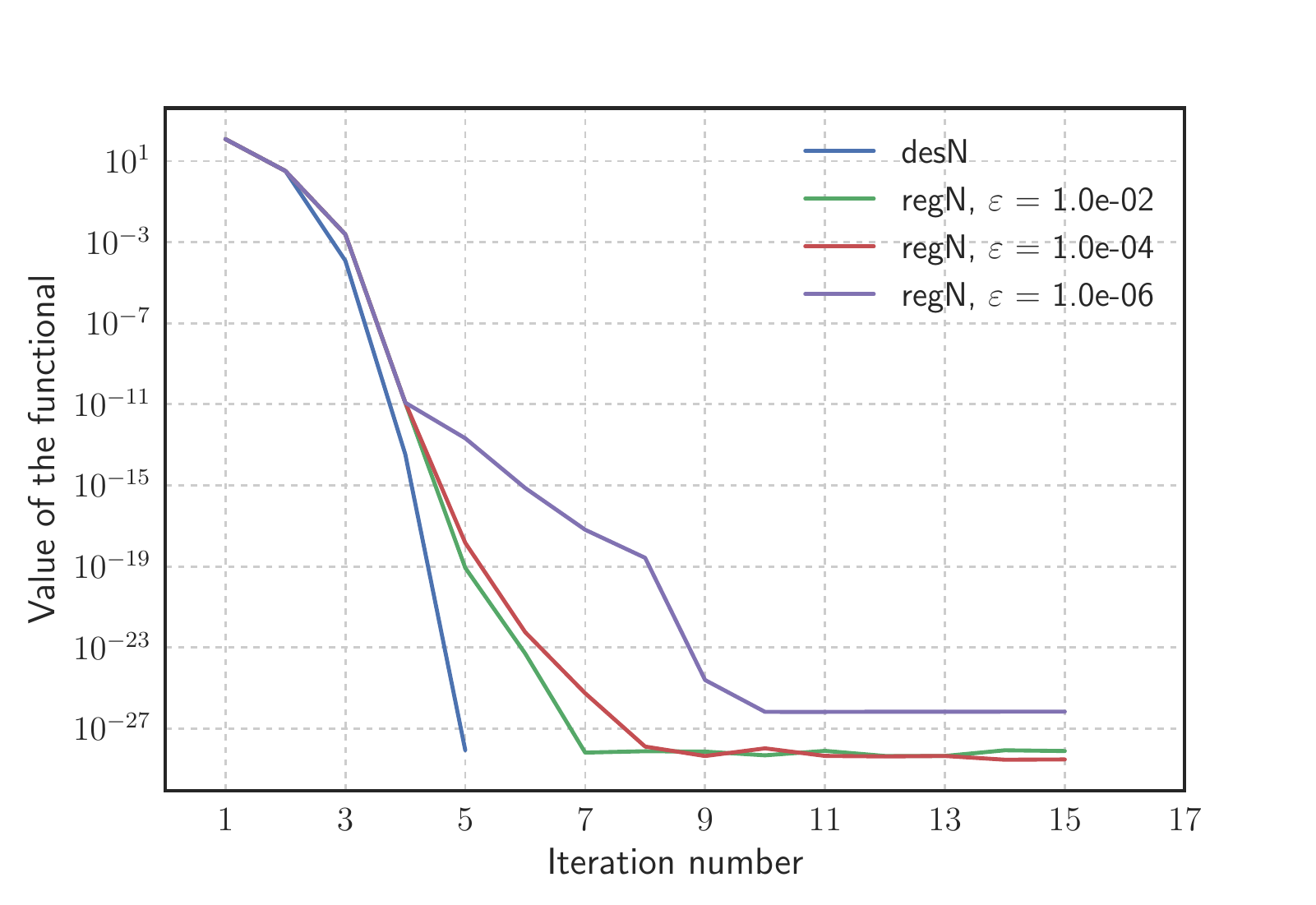}} 
	\caption{Matrix completion tests in the case of strong rank deficiency.}
	\label{fig:ill-compl}
\end{figure}
Results are given in \cref{fig:ill-compl}. We see that even though in all the cases value of the functional goes to $0$, regularized Newton method fails to recover that $\sigma_2$ of the exact answer is in fact $0$ and it's behavior depends on the value of $\varepsilon$.

\section{Related work}
\label{sec:related}
Partly similar approach using so-called \emph{parametrization via kernel} is described in \cite{markovsky2011low, markovsky2013structured}. However, optimization algorithm proposed there is not considered as an optimization problem on a manifold of tuples $(A,Y)$ and is based on two separate optimization procedures (with respect to $A$ and to $Y$, where the latter belongs to the orthogonal Stiefel manifold), thus separating the variables. As stated in \cite{markovsky2013structured} in general it has $O(m^3)$ complexity per iteration.
An overview of Riemannian optimization is presented in \cite{qi2016numerical}.
An example of the traditional approach to bounded-rank matrix sets using Stiefel manifolds is given in \cite{koch2007dynamical} where explicit formulas for projection onto the tangent space are presented. An application of Riemannian optimization to low-rank matrix completion where $\Mr$ is considered as a subvariety in the set of all matrices is given in 
\cite{vandereycken2013low}. The case of $F$ being non-smooth but only Lipschitz is studied in \cite{hosseini2016riemannian}. Theoretical properties of matrix completion such as when exact recovering of the matrix is possible are studied in \cite{candes2010power}. Standard references for introductory algebraic geometry are \cite{hartshorne2013algebraic} and \cite{shafarevich1977basic}. For more computational aspects see \cite{grayson2002macaulay}.
\section*{Acknowledgements}
We are grateful to the anonymous referees for their thorough review of our paper and helpful comments.
This study was supported by the Ministry of Education and Science of the Russian Federation
(grant 14.756.31.0001), by RFBR grants 16-31-60095-mol-a-dk, 16-31-00372-mol-a and by Skoltech NGP program.

\bibliography{notsingular-final}

\begin{thebibliography}{10}

\bibitem{absil2009optimization}
{\sc P.-A. Absil, R.~Mahony, and R.~Sepulchre}, {\em Optimization algorithms on
  matrix manifolds}, Princeton University Press, 2009.

\bibitem{absil2012projection}
{\sc P.-A. Absil and J.~Malick}, {\em Projection-like retractions on matrix
  manifolds}, SIAM Journal on Optimization, 22 (2012), pp.~135--158.

\bibitem{benzi2005numerical}
{\sc M.~Benzi, G.~H. Golub, and J.~Liesen}, {\em Numerical solution of saddle
  point problems}, Acta numerica, 14 (2005), pp.~1--137.

\bibitem{candes2010power}
{\sc E.~J. Cand{\`e}s and T.~Tao}, {\em The power of convex relaxation:
  Near-optimal matrix completion}, IEEE Transactions on Information Theory, 56
  (2010), pp.~2053--2080.

\bibitem{davis1970rotation}
{\sc C.~Davis and W.~M. Kahan}, {\em The rotation of eigenvectors by a
  perturbation. iii}, SIAM Journal on Numerical Analysis, 7 (1970), pp.~1--46.

\bibitem{decker1980newton}
{\sc D.~Decker and C.~Kelley}, {\em Newton’s method at singular points. ii},
  SIAM Journal on Numerical Analysis, 17 (1980), pp.~465--471.

\bibitem{dutta2017problem}
{\sc A.~Dutta and X.~Li}, {\em On a problem of weighted low-rank approximation
  of matrices}, arXiv preprint arXiv:1511.00649v3,  (2017).

\bibitem{grayson2002macaulay}
{\sc D.~R. Grayson and M.~E. Stillman}, {\em Macaulay 2, a software system for
  research in algebraic geometry}, 2002.

\bibitem{griffiths2014principles}
{\sc P.~Griffiths and J.~Harris}, {\em Principles of algebraic geometry}, John
  Wiley \& Sons, 2014.

\bibitem{harper2016movielens}
{\sc F.~M. Harper and J.~A. Konstan}, {\em The movielens datasets: History and
  context}, ACM Transactions on Interactive Intelligent Systems, 5 (2016),
  p.~19.

\bibitem{hartshorne2013algebraic}
{\sc R.~Hartshorne}, {\em Algebraic geometry}, Springer Science \& Business
  Media, 2013.

\bibitem{hosseini2016riemannian}
{\sc S.~Hosseini and A.~Uschmajew}, {\em A {R}iemannian gradient sampling
  algorithm for nonsmooth optimization on manifolds}, SIAM Journal on
  Optimization, 27 (2017), pp.~173--189.

\bibitem{koch2007dynamical}
{\sc O.~Koch and C.~Lubich}, {\em Dynamical low-rank approximation}, SIAM
  Journal on Matrix Analysis and Applications, 29 (2007), pp.~434--454.

\bibitem{kressner2016preconditioned}
{\sc D.~Kressner, M.~Steinlechner, and B.~Vandereycken}, {\em Preconditioned
  low-rank riemannian optimization for linear systems with tensor product
  structure}, SIAM Journal on Scientific Computing, 38 (2016),
  pp.~A2018--A2044.

\bibitem{lakshmibai2015grassmannian}
{\sc V.~Lakshmibai and J.~Brown}, {\em The Grassmannian Variety}, Springer,
  2015.

\bibitem{lee-introman-2001}
{\sc J.~M. Lee}, {\em Introduction to smooth manifolds}, 2001.

\bibitem{markovsky2011low}
{\sc I.~Markovsky}, {\em Low rank approximation: algorithms, implementation,
  applications}, Springer Science \& Business Media, 2011.

\bibitem{markovsky2013structured}
{\sc I.~Markovsky and K.~Usevich}, {\em Structured low-rank approximation with
  missing data}, SIAM Journal on Matrix Analysis and Applications, 34 (2013),
  pp.~814--830.

\bibitem{naldi2015exact}
{\sc S.~Naldi}, {\em Exact algorithms for determinantal varieties and
  semidefinite programming}, PhD thesis, Toulouse, INSA, 2015.

\bibitem{qi2016numerical}
{\sc C.~Qi}, {\em Numerical optimization methods on {R}iemannian manifolds},
  PhD thesis, 2011.

\bibitem{shafarevich1977basic}
{\sc I.~R. Shafarevich and K.~A. Hirsch}, {\em Basic algebraic geometry},
  Springer, 1977.

\bibitem{JMLR:v17:16-177}
{\sc J.~Townsend, N.~Koep, and S.~Weichwald}, {\em Pymanopt: A python toolbox
  for optimization on manifolds using automatic differentiation}, Journal of
  Machine Learning Research, 17 (2016), pp.~1--5,
  \url{http://jmlr.org/papers/v17/16-177.html}.

\bibitem{vandereycken2013low}
{\sc B.~Vandereycken}, {\em Low-rank matrix completion by riemannian
  optimization}, SIAM Journal on Optimization, 23-2 (2013), pp.~1214--1236.

\bibitem{yuan2000review}
{\sc Y.-x. Yuan}, {\em A review of trust region algorithms for optimization},
  ICIAM, 99 (2000), pp.~271--282.

\end{thebibliography}
\bibliographystyle{siamplain}
\end{document}